\pdfoutput=1
\newif\ifpersonal
\documentclass[11pt,a4paper,dvipsnames,x11names]{amsart}

\usepackage{amsmath,amsthm,amssymb,mathrsfs,mathtools,tensor,eucal,thmtools} 
\usepackage[all,cmtip]{xy}

\usepackage[LGR,T1]{fontenc} 
\usepackage[utf8]{inputenc} 
\usepackage{CJKutf8}

\usepackage{microtype,inconsolata} 
\usepackage{enumerate,comment,braket,xspace,tikz,tikz-cd,csquotes} 
\usetikzlibrary{shapes.geometric}
\usepackage[centering,vscale=0.7,hscale=0.7]{geometry}
\usepackage[hidelinks,colorlinks=true,linkcolor=Red4,citecolor=RoyalBlue]{hyperref}
\usepackage[capitalize]{cleveref}
\usepackage{xcolor}
\usepackage{xpatch}
\usepackage{tikzarrows}

\usepackage{mathpazo}
\usepackage{euler}

\usepackage{stackrel}
\usepackage{faktor}
\usepackage{appendix}
\usetikzlibrary{decorations.markings}

\numberwithin{equation}{subsection}
\theoremstyle{plain}
\newtheorem{theorem}[equation]{Theorem}
\newtheorem{lemma}[equation]{Lemma}
\newtheorem*{example*}{Example}
\newtheorem*{definition*}{Definition}
\newtheorem{proposition}[equation]{Proposition}

\newtheorem{corollary}[equation]{Corollary}

\theoremstyle{definition}
\newtheorem{definition}[equation]{Definition}
\newtheorem{notation}[equation]{Notation}
\newtheorem{example}[equation]{Example}
\newtheorem{remark}[equation]{Remark}

\newtheorem{recollection}[equation]{Recollection}
\newtheorem{warning}[equation]{Warning}
\newtheorem{construction}[equation]{Construction}

\usepackage{xr}
\ifpersonal
\newcommand{\personal}[1]{\textcolor[rgb]{0,0,1}{(Personal: #1)}}
\newcommand{\discussion}[1]{\textcolor{violet}{(Discussion: #1)}}
\else
\newcommand{\personal}[1]{\ignorespaces}
\newcommand{\discussion}[1]{\ignorespaces}
\fi

\def\I{\mathcal{I}}
\def\G{\mathbb{G}}
\def\N{\mathbb{N}}

\def\S{\mathbb{S}}
\def\ZZ{\mathbb{Z}}

\def\C{\mathcal{C}}
\def\D{\mathcal{D}}
\def\E{\mathcal{E}}
\def\AA{\mathcal{A}}
\def\X{\mathcal{X}}

\def\M{\mathcal{M}}
\def\OO{\mathcal{O}}
\def\S{\textsf{S}}
\def\Repk{\Rep_k(\I)}
\def\dREP{\mathbf{Rep}_k(\I)^{[0,0]}}
\def\REP{\mathbf{Rep}_k(\I)^\heartsuit}
\def\Prlo{\mathrm{Pr}^{\mathrm{L}, \omega}}
\def\taust{\tau^{st}}
\def\ST{\overline{\mathrm{ST}}_R}
\def\Funst{\mathrm{Fun}^{\mathrm{st}}}
\def\FunR{\mathrm{Fun}^{\mathrm{R}}}

\DeclareMathOperator{\Rep}{Rep}
\DeclareMathOperator{\QCoh}{QCoh}
\DeclareMathOperator{\dAff}{dAff}
\DeclareMathOperator{\dSt}{dSt}

\DeclareMathOperator{\Aff}{Aff}

\DeclareMathOperator{\Hom}{Hom}
\DeclareMathOperator{\Cons}{Cons}

\DeclareMathOperator{\Fun}{Fun}
\DeclareMathOperator{\Sh}{Sh}

\DeclareMathOperator{\hyp}{hyp}

\DeclareMathOperator{\Mod}{Mod}
\DeclareMathOperator{\LMod}{LMod}

\DeclareMathOperator{\St}{St}

\DeclareMathOperator{\Exit}{Exit}

\DeclareMathOperator{\Map}{Map}

\DeclareMathOperator{\End}{End}
\DeclareMathOperator{\Cat}{\mathrm{Cat}}
\DeclareMathOperator{\Spc}{Spc}
\DeclareMathOperator{\ev}{ev}
\DeclareMathOperator{\coev}{coev}
\DeclareMathOperator{\Id}{Id}

\DeclareMathOperator{\Perf}{Perf}

\DeclareMathOperator{\Spec}{Spec}

\makeatletter
\renewenvironment{proof}[1][\relax]{\par
  \pushQED{\qed}%
  \normalfont \topsep6\p@\@plus6\p@\relax
  \trivlist
  \item[\hskip\labelsep\itshape
    \ifx#1\relax \proofname\else\proofname{} of #1\fi\@addpunct{.}]\ignorespaces
}{%
  \popQED\endtrivlist\@endpefalse
}
\makeatother

\makeatletter
\let\save@mathaccent\mathaccent
\newcommand*\if@single[3]{%
	\setbox0\hbox{${\mathaccent"0362{#1}}^H$}%
	\setbox2\hbox{${\mathaccent"0362{\kern0pt#1}}^H$}%
	\ifdim\ht0=\ht2 #3\else #2\fi
}
\newcommand*\rel@kern[1]{\kern#1\dimexpr\macc@kerna}
\newcommand*\widebar[1]{\@ifnextchar^{{\wide@bar{#1}{0}}}{\wide@bar{#1}{1}}}
\newcommand*\wide@bar[2]{\if@single{#1}{\wide@bar@{#1}{#2}{1}}{\wide@bar@{#1}{#2}{2}}}
\newcommand*\wide@bar@[3]{%
	\begingroup
	\def\mathaccent##1##2{%
		\let\mathaccent\save@mathaccent
		\if#32 \let\macc@nucleus\first@char \fi
		\setbox\z@\hbox{$\macc@style{\macc@nucleus}_{}$}%
		\setbox\tw@\hbox{$\macc@style{\macc@nucleus}{}_{}$}%
		\dimen@\wd\tw@
		\advance\dimen@-\wd\z@
		\divide\dimen@ 3
		\@tempdima\wd\tw@
		\advance\@tempdima-\scriptspace
		\divide\@tempdima 10
		\advance\dimen@-\@tempdima
		\ifdim\dimen@>\z@ \dimen@0pt\fi
		\rel@kern{0.6}\kern-\dimen@
		\if#31
		\overline{\rel@kern{-0.6}\kern\dimen@\macc@nucleus\rel@kern{0.4}\kern\dimen@}%
		\advance\dimen@0.4\dimexpr\macc@kerna
		\let\final@kern#2%
		\ifdim\dimen@<\z@ \let\final@kern1\fi
		\if\final@kern1 \kern-\dimen@\fi
		\else
		\overline{\rel@kern{-0.6}\kern\dimen@#1}%
		\fi
	}%
	\macc@depth\@ne
	\let\math@bgroup\@empty \let\math@egroup\macc@set@skewchar
	\mathsurround\z@ \frozen@everymath{\mathgroup\macc@group\relax}%
	\macc@set@skewchar\relax
	\let\mathaccentV\macc@nested@a
	\if#31
	\macc@nested@a\relax111{#1}%
	\else
	\def\gobble@till@marker##1\endmarker{}%
	\futurelet\first@char\gobble@till@marker#1\endmarker
	\ifcat\noexpand\first@char A\else
	\def\first@char{}%
	\fi
	\macc@nested@a\relax111{\first@char}%
	\fi
	\endgroup
}
\makeatother

\usetikzlibrary{decorations.markings} 
\tikzset{
  closed/.style = {decoration = {markings, mark = at position 0.5 with { \node[transform shape, xscale = .8, yscale=.4] {/}; } }, postaction = {decorate} },
  open/.style = {decoration = {markings, mark = at position 0.5 with { \node[transform shape, scale = .7] {$\circ$}; } }, postaction = {decorate} }
}

\makeatletter
\tikzcdset{
  open/.code     = {\tikzcdset{hook, circled};},
  closed/.code   = {\tikzcdset{hook, slashed};},
  open'/.code    = {\tikzcdset{hook', circled};},
  closed'/.code  = {\tikzcdset{hook', slashed};},
  circled/.code  = {\tikzcdset{markwith = {\draw (0,0) circle (.375ex);}};},
  slashed/.code  = {\tikzcdset{markwith = {\draw[-] (-.4ex,-.4ex) -- (.4ex,.4ex);}};},
  markwith/.code ={
    \pgfutil@ifundefined%
    {tikz@library@decorations.markings@loaded}%
    {\pgfutil@packageerror{tikz-cd}{You need to say %
      \string\usetikzlibrary{decorations.markings} to use arrows with markings}{}}{}%
    \pgfkeysalso{/tikz/postaction = {
      /tikz/decorate,
      /tikz/decoration={markings, mark = at position 0.5 with {#1}}}
    }
  },
}
\makeatother

\newcommand{\leftrarrows}{\mathrel{\raise.75ex\hbox{\oalign{%
  $\scriptstyle\leftarrow$\cr
  \vrule width0pt height.5ex$\hfil\scriptstyle\relbar$\cr}}}}
\newcommand{\lrightarrows}{\mathrel{\raise.75ex\hbox{\oalign{%
  $\scriptstyle\relbar$\hfil\cr
  $\scriptstyle\vrule width0pt height.5ex\smash\rightarrow$\cr}}}}
\newcommand{\Rrelbar}{\mathrel{\raise.75ex\hbox{\oalign{%
  $\scriptstyle\relbar$\cr
  \vrule width0pt height.5ex$\scriptstyle\relbar$}}}}

\makeatletter
\def\leftrightarrowsfill@{\arrowfill@\leftrarrows\Rrelbar\lrightarrows}
\newcommand{\xleftrightarrows}[2][]{\ext@arrow 3399\leftrightarrowsfill@{#1}{#2}}
\makeatother

\NewCommandCopy{\notocsection}{\section}
\xpatchcmd{\notocsection}{{1}}{{1001}}{}{}

\begin{document}

\title{Good moduli space for constructible sheaves and Stokes functors}

\author{Enrico Lampetti}
\address{Sorbonne Université and Université Paris Cité, CNRS, IMJ-PRG, F-75005 Paris, France}
\email{enrico.lampetti@imj-prg.fr}

\subjclass[2020]{}

\begin{abstract}
In this paper we give construct good moduli spaces for constructible sheaves and Stokes functors. Derived enhancement of such are also considered.
\end{abstract}

\maketitle


\tableofcontents

\section{Introduction}
	Good moduli spaces for algebraic stacks were introduced by Alper in \cite{Alp} and have played a prominent role in moduli theory as generalizations of Mumford's good GIT quotients \cite{GIT} and Abramovich-Olsson-Vistoli's tame stacks \cite{AOV}.  
	For the sake of this introduction, recall that a good moduli space for an algebraic stack $\X$ is a qcqs morphism $q \colon \X \to X$ to an algebraic space such that the pushforward along $q$ is exact on quasi-coherent sheaves and such that the canonical morphism $\OO_X \to q_\ast \OO_\X$ is an equivalence.
	If a good moduli space $q \colon \X \to X$ exists, then it is unique, as $q$ is universal for maps from $\X$ to algebraic spaces (\cite[Theorem 6.6]{Alp}).		         
    In a similar way that the Keel-Mori Theorem \cite{KM} provides an intrinsic way to show that a Deligne-Mumford stack admits a coarse moduli space, a fundamental result of Alper-Halpern Leistner-Heinloth \cite[Theorem A]{AHLH} provides an intrinsic way to show that an algebraic stack admits a good moduli space.
    The existence of a good moduli space for an algebraic stack $\X$ has many pleasant consequences.
	It allows for example to give a local presentation of $\X$ by quotient stacks and prove the compact generation of $\QCoh(\X)$ (\cite{AHR}). 
	Also, the existence of good moduli spaces can be exploited to construct BPS Lie algebras and unlock the study of the cohomology of $\X$ via cohomological Hall algebras, see e.g. \cite{DHM, Davison, BDNIKP, Hennecart}. 
	In this paper we study existence of good moduli spaces for moduli of representations.
	In particular, we will specialize our discussion to moduli of constructible sheaves and of Stokes data. \medskip

	Let us introduce our main results.
	Consider a compact $\infty$-category $\I \in \Cat_\infty^\omega$ and let $k$ be a noetherian ring of characteristic zero.
	Then we can form the $\infty$-category 
\[
\Repk \coloneqq \Fun(\I; \Mod_k)
\]
of $k$-linear representation of $\I$. This is a finite type $k$-linear category, that is, $\Repk$ is a compact object in $\Prlo_k$.
	Moreover $\Repk$ is equipped with a standard $t$-structure  $\tau_{st}$ for which an object $F \in \Repk$ is (co)connective if and only if $F(c) \in \Mod_k$ is coconnective for every $c \in \I$.
	We then define the moduli of flat representations of $\I$, denoted $\mathbf{Rep}_k(\I)^\heartsuit$, by sending a $k$-algebra $A$ to the maximal grupoid of $\mathrm{Rep}_A(\I) \coloneqq \Fun(\I; \Mod_A)$ spanned by representation that take values in finitely presented flat $A$-modules.
	The moduli $\mathbf{Rep}_k(\I)^\heartsuit$ is an algebraic stack locally of finite presentation over $k$.
	In this setting, and for $\Spec(\kappa) \to \Spec(k)$ a closed point with $\kappa$ algebraically closed, our main result is the following
\begin{theorem}[{\cref{good_Rep_theorem}}]\label{good_Rep_theorem_intro}
	The algebraic stack $\mathbf{Rep}_k(\I)^\heartsuit$ admits a good moduli space $\Repk$ whose $\kappa$-points parametrize semisimple representations of finite dimensional $\kappa$-vector spaces.
\end{theorem}

	The main tool for constructing good moduli spaces is \cite[Theorem A]{AHLH}.
    It states that an algebraic stack of finite presentation over a noetherian ring of characteristic zero admits a separated good moduli space if and only if it is "$\Theta$-reductive" (\cref{def_Theta_reductive}) and "$\mathrm{S}$-complete" (\cref{def_S_complete}).
	These two properties are of valuative nature and involve the algebraic stacks
\[
\Theta_R \coloneqq \left[\faktor{A_{1,R}}{\mathbb{G}_{m,R}}\right], \qquad \ST \coloneqq \left[\faktor{\Spec(R[s,t]/(st-\pi))}{\G_{m,R}}\right],
\]
for $R$ a DVR and $\pi \in R$ a uniformizer.
	For moduli of objects one can explicitly describe $\Theta$-reductiveness and $S$-completeness (\cite{good_Perv}).
	In the setting we focus in the present paper, we can prove that $\mathbf{Rep}_{k}(\I)$ is $\Theta$-reductive and $\mathrm{S}$-complete differently.
	We will show that $\mathbf{Rep}_{k}(\I)$ satisfies a form of \textit{Hartogs' Lemma} (\cref{Hartogs_Rep}).
	The proof follows by an Hartogs type property for locally free sheaves on a regular noetherian ring of dimension $2$.
	Nevertheless, in order to characterize points of the good moduli space as the semisimple representations we need the full strength of \cite{good_Perv}.
	We will then specialize our discussion to the case of constructible sheaves and Stokes data.
	While de former is a direct application of \cref{good_Rep_theorem_intro}, the latter requires some work.
	\medskip

	Let us point out that notion of derived good moduli spaces has been introduced in \cite{derived_good}.
	In particular, \cref{good_Rep_theorem_intro} can be enhanced into a statement about derived stack thanks \cite[Theorem 2.12]{derived_good}.
	Also, \cref{good_Rep_theorem_intro} has some overlap with the work of Fernandez Herrero-Lennen-Makarova \cite{FHLM}, where the case of representations of an acyclic quiver is treated between others. \medskip
	
	Let us now describe the main applications of this paper.

\subsection{Constructible sheaves}

	Constructible sheaves are natural generalization of local systems for stratified space $(X,P)$.
	More precisely, a constructible sheaf is a sheaf that is locally constant when restricted to every stratum.
	Under mild hypothesis on the stratification, a combinatorial description of the category of constructible sheaves is available (\cite{Treumann, HA, Exodromy, exodromyconicality}).
	The work of Haine-Porta-Teyssier \cite{exodromyconicality} shows that the above holds for any algebraic variety equipped with a finite stratification by Zariski locally closed subsets.
	This allows to construct a stack $\mathbf{Cons}^\heartsuit_P(X)$ of constructible sheaves in the same spirit of the character stack: given an algebra $A$ over $k$, the $A$-points of $\mathbf{Cons}^\heartsuit_P(X)$ are constructible sheaves whose stalks are finitely presented flat $A$-modules.
	\cref{good_Rep_theorem_intro} directly applies to the setting of constructible sheaves, yielding the existence of a good moduli space for $\mathbf{Cons}^\heartsuit_P(X)$.

\subsection{Stokes data}
	Stokes data are yet another generalization of local systems. 
Several equivalent description of the category of Stokes data are available, one of which consists in including the datum of a filtration to the stalks of a local system.
	Clearly, conditions are imposed in this filtration.
	While perverse sheaves are linked to the study of differential equation with regular singularities, Stokes data are linked to the study of differential equation with \textit{irregular} singularities.
	This is the irregular Riemann-Hilbert correspondence.
	For $X$ an algebraic variety over $\mathbb{C}$, we can picture the situation as follows: 
	\medskip
	
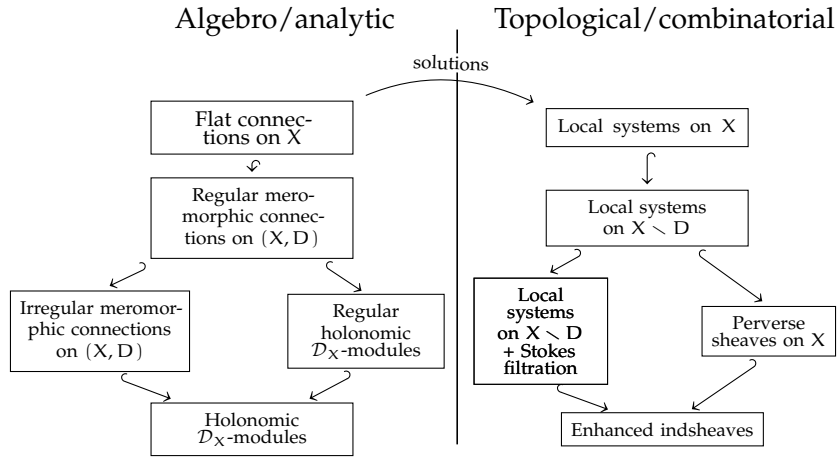
\begin{center}
	\begin{tikzpicture}
		
		
		\draw[line width=0.6pt] (0,4.5) -- (0,-1) ;
		
		\node[anchor=center,align=center,text width=2cm] (AA) at (-2.7,4.6) {Algebro/analytic} ;
		
		\node[anchor=center,align=center,text width=4cm] (TC) at (2.5,4.6) {Topological/combinatorial} ;
		
		\node[text width=2.4cm,align=center,draw,font=\tiny\linespread{0.8}\selectfont] (Flat) at (-2.7,3.2) {\scriptsize Flat connections on $X$} ;
		
		\node[text width=2.4cm,align=center,draw] (Loc) at (2.5,3.2) {\tiny Local systems on $X$} ;
		
		\node[text width=2.4cm,align=center,draw,font=\tiny\linespread{0.8}\selectfont] (RegMer) at (-2.7,2) {Regular meromorphic connections on $(X,D)$} ;
		
		\node[text width=2.4cm,align=center,draw,font=\tiny\linespread{0.8}\selectfont] (LocOpen) at (2.5,2) {Local systems on $X \smallsetminus D$} ;
		
		\node[text width=2.1cm,align=center,draw,font=\tiny\linespread{0.8}\selectfont] (IrregMer) at (-4.7,0.5) {Irregular meromorphic connections on $(X,D)$} ;
		
		\node[text width=1.5cm,align=center,draw,font=\tiny\linespread{0.8}\selectfont] (Stokes) at (1.1,0.5) {Local systems on $X \smallsetminus D$ + Stokes filtration} ;
		
		\node[text width=1.5cm,align=center,draw,font=\tiny\linespread{0.8}\selectfont] (Stokes) at (1.1,0.5) {Local systems on $X \smallsetminus D$ + Stokes filtration} ;
		
		\node[text width=1.8cm,align=center,draw,font=\tiny\linespread{0.8}\selectfont] (RegHol) at (-1.2,0.5) {Regular holonomic $\mathcal D_X$-modules} ;
		
		\node[text width=1.5cm,align=center,draw,font=\tiny\linespread{0.8}\selectfont] (Perv) at (4.1,0.5) {Perverse sheaves on $X$} ;
		
		\node[text width=2.4cm,align=center,draw,font=\tiny\linespread{0.8}\selectfont] (Hol) at (-2.7,-0.8) {Holonomic $\mathcal D_X$-modules} ;
		
		\node[text width=2.4cm,align=center,draw,font=\tiny\linespread{0.8}\selectfont] (Enh) at (2.7,-0.8) {Enhanced indsheaves} ;
		
		\draw (Flat.north east) edge[->,bend left=25,shorten <=5pt,shorten >=5pt] node[midway,sloped,above=3pt,fill=white,inner sep=1]{\tiny solutions} (Loc.north west)  ;
		
		\draw[right hook->,shorten <= 2pt, shorten >= 2pt] (Flat) -- (RegMer) ;
		\draw[right hook->,shorten <= 2pt, shorten >= 2pt] (Loc) -- (LocOpen) ;
		
		\draw[left hook->,shorten <= 8pt, shorten >= 4pt] (RegMer) -- (IrregMer.north) ;
		\draw[left hook->,shorten <= 6pt, shorten >= 4pt] (LocOpen) -- (Stokes.north) ;				
		
		\draw[right hook->,shorten <= 5pt, shorten >= 2pt] (RegMer) -- (RegHol.north) ;
		\draw[right hook->,shorten <= 5pt, shorten >= 3pt] (LocOpen) -- (Perv.north) ;
		
		\draw[right hook->,shorten <= 8pt, shorten >= 4pt] (IrregMer.south) -- (Hol) ;
		\draw[left hook->,shorten <= 6pt, shorten >= 3pt] (RegHol.south) -- (Hol) ;
		\draw[right hook->,shorten <= 6pt, shorten >= 3pt] (Stokes.south) -- (Enh) ;
		\draw[left hook->,shorten <= 5pt, shorten >= 2pt] (Perv.south) -- (Enh) ;

		
	\end{tikzpicture}
\end{center}

\bigskip

\noindent The passage from left to right is given by taking solutions.
	For instance, if $(V,\nabla)$ is a flat connection on a complex analytic variety $X$, then
\[ \mathsf{Sol}(V,\nabla) = \{s \in V \mid \nabla(s) = 0 \} \]
forms a local system (ultimately, an instance of Cauchy's theorem on ordinary differential equations).
	Remarkably, this process can be inverted, giving rise to an equivalence of categories.
Deligne showed that the first two lines from the top are related by an equivalence of this kind \cite{Deligne}.
	Later, Kashiwara extended this equivalence to regular holonomic $\mathcal{D}$-modules and perverse sheaves \cite{Kashiwara}, whereas Deligne and Malgrange first, and Mochizuki after \cite{Mochizuki_Wild} dealt with irregular connections and Stokes filtered local systems.
	Finally, D'Agnolo and Kashiwara \cite{DK} managed to obtain a fully faithful solution functor for all holonomic $\mathcal{D}$-modules. \medskip

	Stacks of Stokes data and their good moduli spaces have been constructed and intensively studied in dimension one in \cite{BB, Boalch1, Boalch2, Boalch3, DDP, HMW}.
	Good moduli spaces in \textit{loc.cit} are constructed via GIT methods.
	For an overview of the theory of Stokes data in dimension $1$ and of its application, let us refer the reader to Boalch's \cite{Boalch_HDR}.
	Porta-Teyssier have constructed the stacks of Stokes data in any dimension and equip it with a natural derived enhancement in \cite{Geometric_Stokes}.
	The moduli of Stokes data is not a moduli of representations: it parametrizes objects in a full subcategory of a category of representations.
	This prevent the application of \cref{good_Rep_theorem_intro}, but we will show that the same proof essentially goes through. \medskip
	
	Before stating our main theorem, let us introduce a bit of notation.
	Let $k$ be a discrete noetherian ring of characteristic zero.
	Let $(X,D)$ be a strict normal crossing pair admitting a smooth compactification.
	Let $\mathscr{I}$ be a good sheaf of unramified irregular values \cref{goodness_ramified_case}.
	Associated to $\mathscr{I}$ there is a cocartesian fibration in finite posets (\cref{cocartesian_fibration}).
	We denote by $\I$ the total space of this cocartesian fibration.
	For a $k$-algebra $A$, we denote by $\St_{\I, A} \subset \Rep_A(\I)$ the full subcategory spanned by punctually split and cocartesian functors  (\cref{def_PS} and \cref{def_cocart}).
	$\St_{\I, A}$ is the category of $A$-linear Stokes functors.
	We can form a prestack $\mathbf{St}^\heartsuit_{\I}$ that assign to a $k$-algebra $A$ the $\infty$-grupoid spanned by Stokes functors that take values in finitely presented flat $A$-modules.
	The stack $\St_{\I, A}$ is an algebraic stack locally of finite presentation over $k$.
	In this setting, for $\Spec(\kappa) \to \Spec(k)$ a closed point with $\kappa$ algebraically closed, our main result is the following
	
\begin{theorem}\label{good_moduli_Stokes_theorem_intro}
	The algebraic stack $\mathbf{St}^{\heartsuit}_{\I,k}$ admits a separated good moduli space $\mathrm{St}^{\heartsuit}_{\I,k}$ whose $\kappa$-points parametrize pseudo-perfect semisimple objects of $\St_{\I, \kappa}^\heartsuit$.
\end{theorem}

	\cref{good_moduli_Stokes_theorem_intro} can be enhanced into a statement about derived stack thanks \cite[Theorem 2.12]{derived_good}.

\subsection{Linear overview}
	In \cref{generalities} we recall the definitions and results about moduli of objects and good moduli spaces we will need in the rest of the paper. \medskip
	
	In \cref{good_Rep_section} we prove that the moduli of representations satisfy the Hartogs' principle, thus yielding their $\Theta$-reductiveness and $\mathrm{S}$-completeness.
	We then show that moduli of representations have quasi-compact connected components.
	The combination of these two results yields the existence of good moduli spaces for moduli of representations. \medskip
	
	In \cref{Stokes_section} we recall the definition of Stokes functors and prove the existence of good moduli spaces for Stokes functors along the lines of \cref{good_Rep_section}.

\subsection{Notation}
We introduce the following running notations.
\begin{itemize}\itemsep=0.2cm
    \item We fix $k$ a discrete noetherian ring of characteristic $0$, that is, a discrete noetherian ring containing $\mathbb{Q}$; 
    \item $\mathrm{Grpd}$ is the category of groupoids;
    \item $\mathrm{Spc}$ is the $\infty$-category of spaces (a.k.a., $\infty$-grupoids);
    \item $\Cat_\infty$ is the $\infty$-category of $\infty$-categories;
    \item $\Cat_\infty^{\omega} \subset \Cat_\infty$ is the full subcategory spanned by compact objects;
    \item $(-)^{\simeq} \colon \Cat_\infty \to \mathrm{Spc}$ is the functor that assign to an $\infty$-category its maximal sub-$\infty$-grupoid;
    \item $\Pr^{\mathrm{L}}_k$ is the $\infty$-category of presentable $k$-linear $\infty$-categories with left adjoints functors;
    \item $\Prlo_k$ is the $\infty$-category of presentable compactly generated $k$-linear $\infty$-categories with left adjoints functors;
    \item $\Aff_k$ is the opposite of the category of $k$-algebras;
    \item $\dAff_k$ is the opposite of the $\infty$-categories of simplicial $k$-algebras.
    \item the category of stacks over $k$ is the full subcategory
    \[
    \mathrm{St}_k \subset \Fun(\Aff_k, \mathrm{Grpd})
    \]
    spanned by derived stacks over $k$;
    \item the $\infty$-category of derived stacks over $k$ is the full subcategory
    \[
    \dSt_k \subset \Fun(\dAff_k, \mathrm{Spc})
    \]
    spanned by derived stacks over $k$;
    \item denote by $\tau_{\geq 1}: \mathrm{Spc} \to \mathrm{Grpd}$ the functor that assign to a space its homotopy category, which is a groupoid.
    Denote by $i: \Aff_k \to \dAff_k$ the canonical inclusion.
    The composition 
    \[
    \Fun(\dAff_k, \mathrm{Spc}) \xrightarrow{- \circ i} \Fun(\Aff_k, \mathrm{Spc}) \xrightarrow{\tau_{\geq 1} \circ -} \Fun(\Aff_k, \mathrm{Grpd})
    \]
    restricts to a functor $t_0 \colon \dSt_k \to \St_k$.
    We will refer to $t_0$ as the classical truncation functor.
\end{itemize}

\section{Moduli of objects and $t$-structures}\label{generalities}

\subsection{Families of objects}\label{section_moduli_of_objects}
In this paragraph we recall properties of tensor product of presentable $\infty$-categories, which is used to define families of objects. 
We also discuss some non-commutative finiteness properties one can impose.

\begin{recollection}[{\cite[Section 4.8]{HA}}]\label{tensor_product}
	For $\C, \D \in \mathrm{Pr}^\mathrm{L}_k$, we can consider their tensor product
\[
\C \otimes_k \D \simeq \FunR_k(\C^{op}, \D)
\]
where the right hand side is the $\infty$-category of $k$-linear functors commuting with limits.
	This tensor product endows $\Pr^{\mathrm{L}}_k$ with a symmetric monoidal structure which restricts to $\Prlo_k$. 
	When $\C$ is compactly generated, there is a canonical equivalence
\[
\C \otimes_k \D \simeq \Funst_k((\C^{\omega})^{op}, \D)
\]
where the right hand side is the $\infty$-category of exact $k$-linear functors.
\end{recollection}

\begin{remark}\label{functoriality_tens_prod}
	Let $\C, \D, \E \in \Pr_k^{\mathrm{L}}$.
	Then an adjunction
\[
f \colon \D \leftrightarrows \E \colon g
\]
induces an adjunction 
\[
\Id_\C \otimes_k f \colon \C \otimes_k \D\leftrightarrows \C \otimes_k \E \colon \Id_\C \otimes_k g.
\]
	The functor $\Id_\C \otimes_k g$ corresponds to 
\[
g \circ - \colon \FunR_k(\C^{op}, \E) \to C \otimes_k \E \simeq \FunR_k(\C^{op}, \D)
\]
under the equivalences of \cref{tensor_product}.
	Its left adjoint $\Id_\C \otimes_k f$ has no easy description in general.
	Notice indeed that post-composition with $f$ do not preserve functors that commutes with limits.
	Nevertheless, if $\C$ is compactly generated one can see that the adjunction $\Id_\C \otimes_k f  \dashv \Id_\C \otimes_k g$ corresponds to 
\[
f \circ - \colon \Funst_k((\C^{\omega})^{op}, \D) \leftrightarrows \Funst_k((\C^{\omega})^{op}, \D) \colon g \circ -
\]
under the equivalence of \cref{tensor_product}.
\end{remark}

\begin{notation}
In the setting of \cref{functoriality_tens_prod}, we denote the functors $\Id_\C \otimes_k f, \Id_\C \otimes_k g$ respectively by $f_\C, g_\C$.
\end{notation}

	In the stable setting, the following result is a reformulation of \cite[Lemma 2.2.1]{SS}.

\begin{lemma}[{\cite[Lemma 2.4]{AG}}]\label{generators_stable}
	Let $\E \in \Prlo$ stable and $X \subset \E^\omega$ a set of compact objects.
	Then the following are equivalent:
\begin{enumerate}\itemsep=0.2cm
	\item $\forall y \in \E$,
\[
\mathrm{Map}_\E(x,y) \simeq * \in \mathrm{Sp} \ \  \forall x \in X \Rightarrow y \simeq 0;
\]
	\item the $\infty$-category $\E^\omega$ is the smallest full subcategory of $\E$ containing $X$ and stable under finite colimits, shifts and retracts.
\end{enumerate}
\end{lemma}

\begin{definition}\label{compact_gen_def}
	Let $\E \in \Prlo$ stable and $X \subset \E^\omega$ a set of compact objects.
	We say that $X$ is a set of compact generators if it satisfies the equivalent conditions of \cref{generators_stable}.
\end{definition}

\begin{example}
	Let $A \in \mathrm{Alg}_k$.
	Then $A$ is a compact generator of $\Mod_A$ in the sense of \cref{compact_gen_def}.
\end{example}

\begin{definition}
Let $\C \in \Pr^\mathrm{L}_k$ and $A \in \dAff_k$.
We define the $\infty$-category of $A$-families of objects in $\C$ as
\[
\C_A \coloneqq  \C \otimes_k \Mod_A.
\]
\end{definition}

\begin{definition}
Let $\I, \E \in \Cat_\infty$.
The $\infty$-category of representations of $\I$ with coefficients in $\E$ is
\[
\Rep_{\E}(\I) \coloneqq \Fun(\I, \E) \ .
\]
\end{definition}

\begin{notation}
Let $\I \in \Cat_\infty$ and $\Spec A \in \dAff_k$.
Then we use the shortcut notation $\Rep_A(\I)$ for $\Rep_{\Mod_A}(\I)$.
Notice that $\Rep_A(\I) \in \Prlo_k$.
\end{notation}

\begin{remark}\label{tens_prod_Rep}
Let $\I \in \Cat_\infty^\omega$. For $\C = \Repk$, we have 
\begin{align*}
\Repk_A & \simeq \Fun(\I, \Mod_k) \otimes_k \Mod_A \simeq \Fun(\I, \Mod_k \otimes_k \Mod_A) \simeq \Fun(\I, \Mod_A) = \\
& = \Rep_A(\I).
\end{align*}
\end{remark}

\begin{definition}\label{pseudo_perfect}
Let $\C \in \Prlo_k$ and  $\Spec(A) \in \dAff_k$.
We define the $\infty$-category of pseudo-perfect families over $A$ in $\C$ as the full subcategory
\[
\Funst((C^\omega)^{op}, \Perf(A)) \subset \C_A
\]
spanned by functors valued in $\Perf(A)$.
\end{definition}

\begin{remark}
Let $\I \in \Cat_\infty^\omega$ and $\Spec(A) \in \dAff_k$.
By \cref{tensor_product} and \cref{tens_prod_Rep} we have a chain of equivalences
\[
\Rep_A(\I) \simeq \Repk_A \simeq \Fun^{st}_k((\Repk^\omega)^{op}, \Mod_A).
\]
Under the above equivalences, a representation $F \in \Rep_A(\I)$ corresponds to the (restricted) enriched Yoneda functor
\[
h_F \colon (\Repk^\omega)^{op} \to \Mod_A
\]
\[
G \to \Hom_{\Rep_A(\I)}(G \otimes_k A, F)
\]
\end{remark}

In order to introduce finiteness conditions on presentable categories, let us give the following:

\begin{lemma}[{\cite[Theorem D.7.0.7]{SAG}}]\label{finiteness_conditions}
Let $\C \in \Prlo_k$.
Then $\C$ is dualizable in $\Pr_k^{\mathrm{L}}$ with dual given by $\C^\vee \simeq \mathrm{Ind}((\C^\omega)^{op})$.
In particular, there are up to a contractible space of choices unique maps in $\Pr^{\mathrm{L}}_k$
\[
\mathrm{ev}_\C \colon  \C^\vee \otimes_k \C \to \Mod_k
\]
\[
\mathrm{coev}_\C  \colon  \Mod_k \to \C \otimes_k \C^\vee
\]
satisfying the triangular identities.
\end{lemma}

\begin{remark}
Let $\I \in \Cat_\infty^\omega$. For $\C = \Repk \coloneqq \Fun(\I, \Mod_k)$ we have $\C^\vee \simeq \Fun(\I^{op}, \Mod_k)$.
\end{remark}

	We will need the following finiteness condition:

\begin{definition}\label{def_finiteness_conditions}
	Let $\C \in \Prlo_k$.
	We say that $\C$ is 
\begin{enumerate}\itemsep=0.2cm
    \item smooth if $\ev_\C$ preserves compact objects;
    \item proper if $\coev_\C$ preserve compact objects;
    \item of finite type if it is a compact object of $\Prlo_{k}$.
\end{enumerate}
\end{definition}

	The next result compares the different finiteness conditions of \cref{def_finiteness_conditions}.

\begin{proposition}[{\cite[Proposition I.4.16]{HDR}}]\label{relation_finiteness_conditions}
	Let $\C \in \Prlo_k$.
	Then:
\begin{enumerate}\itemsep=0.2cm
    \item if $\C$ is of finite type, it is smooth;
    \item if $\C$ is smooth and proper, it is of finite type.
\end{enumerate}
Moreover, if $\C$ is smooth then it admits a compact generator in the sense of \cref{compact_gen_def}.
\end{proposition}

\begin{lemma}\label{lemma_smoothness}
	Let $\C \in \Prlo_k$.
	The following are equivalent:
\begin{enumerate}
	\item $\C$ is smooth;
	\item there exists an equivalence $\C \simeq \LMod_A$ with $A \in \mathrm{Alg}_k$ smooth in the sense of \cite[Definition 4.6.4.13]{HA}.
\end{enumerate}
\end{lemma}

\begin{proof}
This follows by combining \cite[Proposition 11.3.2.4]{SAG} with \cite[Remark 4.2.1.37 \& Remark 4.6.4.15]{HA}.
\end{proof}

\begin{remark}\label{rem_smooth}
	Let $\C$ be smooth and $E \in \C$ be a compact generator.
	Lurie shows in the proof of \cite[Proposition 11.3.2.4]{SAG} that $A \coloneqq \End(E) \in \mathrm{Alg}_k$ is smooth and $\C \simeq \LMod_{A^{rev}}$, where $A^{rev}$ is the opposite algebra of $A$.
\end{remark}

\begin{lemma}[{\cite[Corollary I.5.7]{HDR}}]\label{pseudo_perfect_vs_compact}
Let $\C \in \Prlo_k$ and let $\Spec(A) \in \dAff_k$.
\begin{enumerate}\itemsep=0.2cm
    \item If $\C$ is smooth, then any pseudo-perfect object in $\C_A$ is compact.
    \item If $\C$ is proper, then any compact object in $\C_A$ is pseudo-perfect.
\end{enumerate}
\end{lemma}

\begin{corollary}\label{compact_generator_smooth}
	Let $\C \in \Prlo_k$ be a smooth category and $E \in \C$ be a compact generator.
	Let $\Spec (A) \in \dAff_k$.
	Then $F \in \C_A$ is pseudo-perfect if and only if $F(E) \in \Perf(A)$.
\end{corollary}

\begin{proof}
	Since $E$ generates $\C^\omega$ under shifts, finite colimits and retracts and $\Perf(X)$ is closed in $\QCoh(X)$ under shifts finite (co)limits and retracts, the result follows.
\end{proof}

\begin{recollection}\label{cpt_gen_Rep}
	Let $\I \in \Cat_\infty^\omega$. 
	By \cite[\href{https://kerodon.net/tag/068L}{Proposition 068L}]{kerodon},the $\infty$-category $\I$ is a retract of a finite $\infty$-category.
	Hence $\I$ has a finite number of equivalence classes.
	Let $\left\lbrace c_a \right\rbrace_{a \in I}$ be a set of representants of the equivalence classes of $\I$.
	Denote (abusively) by $c_a \colon \left\lbrace * \right\rbrace \to \I$ be the morphism that choose the element $c_a$.
	Let 
\[
\ev_{c_a} \colon \Repk \to \Mod_k
\]
\[
F \mapsto F(c_a)
\]
be the evaluation functor and consider its left adjoint 
\[
\ev_{c_a,!} \colon \Mod_k \to \Repk
\]
obtained by left Kan extension.
	The family of functors
\[
\left\lbrace \Hom_{\Repk}(ev_{c_a,!}, -) \simeq \ev_{c_a} \colon \Repk \to \Mod_k \right\rbrace_{a \in I}
\]
is jointly conservative.
	In particular, $\bigoplus_{a} \ev_{c_a, !}$ is a compact generator of $\Repk$.
\end{recollection}

\begin{corollary}\label{pseudo_perf_Rep}
	In the setting of \cref{cpt_gen_Rep}, let $\Spec(A) \in \dAff_k$.
	A family $F \in \Rep_A(\I) \simeq \Repk_A$ is pseudo-perfect if and only if $F(c_a) \in \Perf(A)$ for every $a \in I$.
\end{corollary}

\begin{proof}
	It follows by \cref{compact_generator_smooth} and \cref{cpt_gen_Rep} that $F \in \Rep_A(\I)$ is pseudo-perfect if and only if
\[
\bigoplus_a \Hom_{\Rep_A(\I)}(\ev_{c_a, !} \otimes_k A, F) \simeq \bigoplus_a F(c_a) \in \Perf(A).
\]
Since a finite direct sum is perfect if and only if each direct summand is perfect, the result follows.
\end{proof}

\subsection{Moduli of objects and $t$-structures}\label{moduli_of_objects}
In this paragraph we introduce the moduli of objects of an $\infty$-category of finite type.
This construction firstly appeared in \cite{TV}, but we will follow the presentation given in \cite[Section 1.5]{HDR}.

\begin{recollection}
	Let $\C \in \Prlo_k$.
	Consider the presheaf
\[
\widehat{\M}_\C \colon \dAff_k^{op} \to \mathrm{Spc}
\]
\[
\Spec(A) \to \C_A^{\simeq}.
\]
	Since $\C$ is dualizable (\cref{finiteness_conditions}) and $\Mod(-)$ satisfies faithfully flat descent (\cite[Corollary D.6.3.3]{HA}), the presheaf $\widehat{\M}_\C$ defines a derived stack over $k$.
	The stack $\widehat{\M}_\C$ is too big and has no chance to be representable.
	The derived moduli of objects of $\C$ is defined as the presheaf 
\[
\M_\C: \dAff_k^{op} \to \mathrm{Spc}
\]
\[
\Spec(A) \to \C_A^{pp}
\]
where $\C_A^{pp} \subset \C_A$ is the maximal groupoid spanned by pseudo-perfect families over $A$ in the sense of \cref{pseudo_perfect}.
Since $\Perf(-)$ satisfies hyper-descent (\cite[Proposition 2.8.4.2-(10)]{SAG}), the sub-presheaf
\[
\M_\C \subset \widehat{\M}_\C
\]
is a substack by \cref{functoriality_tens_prod}.
\end{recollection}

The main theorem of \cite{TV} translates as:
\begin{theorem}[{\cite[Theorem I.5.9]{HDR}}]
Let $\C \in \Prlo_{k}$ be a finite type $\infty$-category.
Then $\M_\C$ is a locally geometric derived stack locally of finite presentation over $k$.
Furthermore, the tangent complex at a point $x \colon  S \to \M_\C$ corresponding to a pseudo-perfect family of objects $M_x \in \C_\S$ is given by
\[
x^*\mathbb{T}_{\M_\C} \simeq \Hom(M_x, M_x)\left[1\right].
\]
\end{theorem}

	In this paper we are mainly interested in moduli of objects of abelian categories.
	If $\C \in \Prlo_k$ is equipped with a $t$-structure $\tau$, we can introduce a substack 
\[
\M_\C^{[0,0]} \subset \M_\C
\] 
whose classical truncation parametrizes families of objects in the heart $\tau$.
	Under suitable assumption on $\tau$, the stack $\M_\C^{[0,0]}$ is an open substack of $\M_\C$, so that it is also locally geometric locally of finite type when $\C$ is a finite type $\infty$-category.

\begin{definition}\label{accessible_t_structure}
	Let $\C \in \Prlo_k$ equipped with a $t$-structure $\tau$. 
	We say that
\begin{enumerate}\itemsep=0.2cm
    \item $\tau$ is $\omega$-accessible if the full subcategory $\C_{\leq 0} \subseteq \C$ is stable under filtered colimits;
    \item $\tau$ is left-complete if $\bigcap \C_{\geq n} \simeq 0$;
    \item $\tau$ is right-complete if $\bigcap \C_{\leq n} \simeq 0$;
    \item $\tau$ is non-degenerate if it is both right-complete and left-complete.
    \item $\tau$ is admissible if it is $\omega$-accessible and non-degenerate.
\end{enumerate}
\end{definition}

\begin{recollection}\label{induced_t_structure}
	Let $\C \in \Prlo_k$.
	Then a morphism $\Spec (B) \to \Spec (A)$ in $\dAff_k$ corresponding to a map $A \to B$ induces an adjunction
\[
- \otimes_A B \colon \C_{A} \leftrightarrows  \C_{B}\colon U
\]
where $U$ is the forgetful functor.
	As explained in \cite[Section I.5.4]{HDR}, given an admissible right-complete $t$-structure $\tau$ on $\C$, we get an induced admissible right-complete $t$-structure $\tau_A$ on $\C_A$ for any $Spec (A) \in \dAff_k$ by declaring that $F \in \C_A$ is (co)connective if and only if $U(F) \in \C$ is (co)connective.
	Moreover $\tau_A$ is non-degenerate if $\tau$ is.\\ \indent
For $X \in \dSt_k$, we can also define an induced $t$-structure $\tau_X$ on $\C_X$ whose connective part is defined by descent.
	That is, $F \in (\C_X)_{\geq 0}$ if and only if $f_\C^\ast F \in (\C_A)_{\geq 0}$ for any $f \colon \Spec (A) \to X$.
\end{recollection}

\begin{definition}\label{def_tau_flat}
	Let $\C \in \Prlo_k$ equipped with an $\omega$-admissible right-complete $t$-structure $\tau$. Let $\Spec (A) \in \dAff_k$.
	We say that a family $F \in \C_A$ is $\tau$-flat over $\Spec (A)$ (or $\tau_A$-flat) if for every $M \in \Mod_A^\heartsuit$ one has $M \otimes_A F \in \C_A^\heartsuit$.
\end{definition}

\begin{remark}\label{tau_flat_in_heart}
	In the setting of \cref{def_tau_flat}, when $\Spec (A)$ is a classical (i.e., underived) affine then a $\tau$-flat family $F$ over $A$ lies in $\C_A^\heartsuit$.
	Indeed in this case one has $A \in \Mod_A^\heartsuit$, so that $F \simeq F \otimes_A A \in \C_A^\heartsuit$.
	The converse holds when $A$ is a field.
\end{remark}

\begin{recollection}\label{standart_t_strucure}
Let $\D \in \Cat_\infty$.
Let $\E \in \Prlo$ be equipped with a $t$-structure $\tau$.
It is immediate to check that the pair
\[(\Fun(D, \E_{\geq 0}), \Fun(D, \E_{\geq 0}))
\]
defines a $t$-structure on $\Fun(D, \E)$.
Moreover $\taust$ is accessible (respectively, right-complete, left-complete, non-degenerate) if $\tau$ is.
\end{recollection}

\begin{definition}
We will refer to the $t$-structure of \cref{standart_t_strucure} as the standard $t$-structure.
We will denote it by $\taust$.
\end{definition}

\begin{remark}
Let $\I \in \Cat_\infty^\omega$ and $\Spec(A) \in \dAff_k$.
Equip $\Repk$ with the standard $t$-structure $\taust$ induced by the canonical $t$-structure on $\Mod_k$.
It follows from the definition of induced $t$-structure that $\tau^{st}_A$ on $\Repk_A$ coincide with the standard $t$-structure on $\Rep_A(\I) = \Fun(\I, \Mod_A)$.
Notice that these $t$-structures are admissible as the canonical $t$-structure on $\Mod_k$ is.
\end{remark}

\begin{remark}\label{tau_flat_Rep}
Let $\I \in \Cat_\infty^\omega$ and $\Spec(A) \in \dAff_k$.
Equip $\Repk$ with the standard $t$-structure $\taust$.
Since the action of $\Mod_A$ on $\Rep_A(\I)$ is computed termwise it follows that $F \in \Rep_A(\I)$ is $\tau_A$-flat if and only if $F(c) \in Mod_A$ has Tor-amplitude $[0,0]$ for every $c \in \I$.
\end{remark}

\begin{lemma}[{\cite[Lemma 2.2.11]{good_Perv}}]\label{pullback_flat_affine}
	Let $\C \in \Prlo_k$ equipped with an accessible $t$-structure $\tau$.
	Let $f\colon \Spec(A) \to \Spec(B)$ in $\dAff_k$.
	If $F \in \C_B$ is $\tau_B$-flat, then $f_\C^*F$ is $\tau_A$-flat.
\end{lemma}

\begin{lemma}[{\cite[Lemma 2.2.15]{good_Perv}}]\label{descent_flat_affine}
Let $\C \in \Prlo_k$ equipped with an admissible $t$-structure $\tau$.
Let $f \colon \Spec(A) \to \Spec(B)$ be a faithfully flat morphism.
Then $F \in \C_B^\heartsuit$ is faithfully flat if and only if $f_\C^*F$ is $\tau_A$-flat.
\end{lemma}

\begin{recollection}\label{descent_flat}
	Given $\C$ in $\Prlo_k$ endowed an admissible $t$-structure $\tau$, we can consider the functor
\[
\M_\C^{[0,0]}: \dAff_k^{op} \to \mathrm{Spc}
\]
sending $\Spec(A) \in \dAff_k$ to the maximal $\infty$-groupoid of $\tau$-flat pseudo-perfect families of objects over $A$.
	By \cref{pullback_flat_affine} and \cref{descent_flat_affine} the functor $\M_\C^{\left[0,0\right]} \subset \M_\C$ is a derived substack.
\end{recollection}

\begin{notation}
We will denote by $\M_\C^\heartsuit$ the classical truncation of the derived stack $\M_\C^{\left[0,0\right]}$. In the special case where $\C = \Repk$ for $\I \in \Cat^\omega$, we will denote these stacks respectively by $\mathbf{Rep}_k(\I)^{\left[0,0\right]}$ and $\mathbf{Rep}_k(\I)^\heartsuit$.
\end{notation}

\begin{remark}
The above notation is justified by the following: when testing on an underived affine, a flat object lies automatically in the heart (see \cref{tau_flat_in_heart}), while in general such an object is only of Tor-amplitude $\left[0,0\right]$.
\end{remark}

The key property for a $t$-structure to define an open substack of $\M_\C$ is linked to the following definition:

\begin{definition}
Let $\Spec(A) \in \dAff_k$ and let $F \in \C_A$.
The flat locus of $M$ is the functor 
\[
\Phi_F\colon  (\dAff_k)^{op}_{/_{\Spec(A)}} \to \Spc
\]
\begin{align*}
(\Spec(B) \xrightarrow{f} \Spec(A)) \mapsto 
\left\{
    \begin{aligned}
        & * \ \ \text{if $f^*(M)$ is $\tau$-flat over $B$,} \\
        & \emptyset \ \ \ \ \text{otherwise.}                  
    \end{aligned}
\right.
\end{align*}
\end{definition}

\begin{definition}\label{opennes_flatness}
Let $\C \in \Pr_k^{L, \omega}$ equipped with an admissible $t$-structure $\tau$.
We say that $\tau$ universally satisfies openness of flatness if for every $Spec(A) \in \dAff_k$ and every $M \in \C_S$ the flat locus $\Phi_F$ of $F$ is an open subscheme of $Spec(A)$.
\end{definition}

\begin{proposition}[{\cite[Proposition I.5.35]{HDR}}]\label{moduli_flat_objects}
Let $\C \in \Pr_k^{L, \omega}$ of finite type equipped with an admissible $t$-structure $\tau$.
If $\tau$ universally satisfies openness of flatness, then the structural morphism
\[
\M_\C^{\left[0,0\right]} \to \M_\C
\]
is representable by a Zariski open immersion.
In particular, $\M_\C^{\left[0,0\right]}$ is a locally geometric derived stack, locally of finite presentation over $k$.
\end{proposition}

	A priori, \cref{moduli_flat_objects} only guarantees that $\M_\C^{[0,0]}$ and its truncation $\M_\C^\heartsuit$ have an open exhaustion by (derived) algebraic stacks.
	When $\Spec (A) \in \dAff_k$ is discrete, the $\infty$-grupoid $\M_\C^{[0,0]}(\Spec(A))$ takes values in $1$-grupoids by \cref{tau_flat_in_heart}.
	Hence $\M_C^\heartsuit$ is $1$-truncated in the sense of \cite{HAG-II}.
	Then \cite[Lemma 2.19]{TV} implies the following

\begin{corollary}\label{geometric_heart}
Let $\C \in \Pr_k^{L, \omega}$ of finite type equipped with an admissible $t$-structure $\tau$.
If $\tau$ universally satisfies openness of flatness, then $\M_\C^{\left[0,0\right]}$ is a $1$-Artin stack.
In particular its truncation $\M_\C^\heartsuit$ is a (classical) algebraic stack locally of finite presentation.
\end{corollary}

\begin{remark}\label{locally_free_valued}
Let $\Spec(A) \in \dAff_k$. Combining \cref{tau_flat_Rep} with \cref{pseudo_perf_Rep}, we get that a family $F \in \Rep_A(\I)$ is $\tau$-flat and pseudo-perfect over $A$ if and only if it takes values in perfect $A$-modules of Tor-amplitude $[0,0]$.
In particular if $A$ is discrete, the family $F$ is $\tau$-flat and pseudo-perfect over $A$ if and only if it takes values in locally free $A$-modules of finite rank.
\end{remark}

\begin{proposition}\label{openness_flatness_Rep}
	Let $\I \in \Cat_\infty^\omega$.
	Then the standard $t$-structure $\taust$ on $\Repk$ satisfies openness of flatness.
	In particular, $\mathbf{Rep}_k(\I)^{\left[0,0\right]}$ is a $1$-Artin derived stack, locally of finite presentation over $k$, and $\mathbf{Rep}_k(\I)^\heartsuit$ is an algebraic stack locally of finite presentation over $k$.
\end{proposition}

\begin{proof}
Let $\left\lbrace c_a \right\rbrace_{a \in I}$ be a finite set of representants of the equivalence classes of $\I$ (\cref{cpt_gen_Rep}).
It follows by \cref{cpt_gen_Rep} and \cref{locally_free_valued} that we have a pullback square
\[\begin{tikzcd}[sep=small]
	{\dREP} && {\mathbf{Rep}_{k}(\I)} \\
	\\
	{\prod_{a \in I}\mathbf{Vect}_k} && {\prod_{a \in I}\mathbf{Perf}_k}
	\arrow[from=1-1, to=1-3]
	\arrow[from=1-1, to=3-1]
	\arrow[from=1-3, to=3-3]
	\arrow[from=3-1, to=3-3]
\end{tikzcd}\]
where the right vertical map is given by the product of evaluations at every $c_a$, $a \in I$, and $\mathbf{Vect}_k \coloneqq \bigsqcup_{n \in \N} \mathrm{B}\mathbf{GL}_n$.
	In particular, the flat locus of $F$ is identified with the (finite) intersection of the flat loci of $F(c_a) \in \Perf(A)$.
	It follows from \cite[Corollary 6.1.4.6]{SAG} that this is an open subscheme of $\Spec(A)$.
\end{proof}

\subsection{Existence criteria for good moduli spaces} \label{existence_criteria}
In this paragraph we recall the definition of good moduli spaces and an existence criterion for them.
References are given to \cite{Alp, AHLH}.

\begin{definition}[{\cite[Definition 4.1]{Alp}}]
Let $q\colon \X \to X$ be a qcqs morphism over an algebraic space $S$ with $\X$ an algebraic stack and $X$ an algebraic space. 
We say that $q \colon \X \to X$ is a good moduli space if the following properties are satisfied:
\begin{enumerate}\itemsep=0.2cm
    \item The functor $q_* \colon \QCoh(\X) \to \QCoh(X)$ is $t$-exact;
    \item we have $q_*\OO_\X \simeq \OO_X$.
\end{enumerate}
\end{definition}

For completeness, we recall some of the main properties of good moduli spaces:

\begin{theorem}[{\cite[Proposition 4.5 \& Theorem 4.16 \& Theorem 6.6]{Alp}}]\label{good_moduli_properties}
Let $q: \X \to X$ be a good moduli space.
Then
\begin{enumerate}\itemsep=0.2cm
    \item The functor $q^*\colon \QCoh(X)^\heartsuit \to \QCoh(\X)^\heartsuit$ is fully faithful;
    \item the map $q$ is universally closed and surjective;
    \item if $Z_1,Z_2$ are closed substacks of $\X$, then
    \[
    \mathrm{Im}(Z_1) \cap \mathrm{Im}(Z_2) = \mathrm{Im}(Z_1 \cap Z_2)
    \]
    where the intersections and images are scheme-theoretic.
    \item for $\overline{k}$ an algebraically closed field over $k$, $\overline{k}$-points of $X$ are $\overline{k}$-points of $\X$ up to closure equivalence (see \cite[Theorem 4.16-(iv)]{Alp} for a precise statement);
    \item If $\X$ is reduced (resp., quasi-compact, connected, irreducible), then $X$ is also.
If $\X$ is locally noetherian
and normal, then $X$ is also.
    \item if $\X$ is locally noetherian, then $X$ is locally noetherian and $q_*\colon \QCoh(\X)^\heartsuit \to \QCoh(X)^\heartsuit$ preserves coherence;
    \item if $k$ is excellent and $\X$ is of finite type, then $X$ is of finite type;
    \item if $\X$ is locally noetherian, then $q \colon \X \to X$ is universal for maps to algebraic spaces.
\end{enumerate}
\end{theorem}

The necessary and sufficient conditions of \cite{AHLH} for the existence of a good moduli spaces are given in terms of valuative criteria for two stacks that we now introduce.
In what follows we fix $R$ a DVR over $k$.
Let $\pi \in R$ be a uniformizer and $\kappa$ the residue field of $R$.
The first stack we have to consider is
\[
\Theta_R\coloneqq \left[\faktor{\Spec(R\left[x\right])}{\mathbb{G}_m}\right],
\]
which has an unique closed point, denoted by $0$.
Explicitly, 
\[
0\coloneqq  \left[\faktor{\Spec(\kappa)}{\mathbb{G}_{m,R}}\right].
\]

\begin{definition}\label{def_Theta_reductive}
We say that $\X \in \mathrm{St}_k$ is $\Theta$-reductive if for any DVR $R$ and any diagram of solid arrows
\[\begin{tikzcd}[sep=small]
	{\Theta_R \setminus \left\lbrace 0 \right\rbrace} && \X \\
	\\
	{\Theta_R}
	\arrow["f", from=1-1, to=1-3]
	\arrow["j"', hook, from=1-1, to=3-1]
	\arrow["{\exists ! \ \widetilde{f}}"', dashed, from=3-1, to=1-3]
\end{tikzcd}\]
there exists a unique extension $\widetilde{f}$ of $f$ making the above triangle commutes.
\end{definition}

The second stack we have to consider is \[
\ST\coloneqq \left[\faktor{\Spec(\faktor{R\left[s,t\right]}{(st-\pi)})}{\mathbb{G}_m}\right],
\]
which has a unique closed point, denoted by $0$.
Explicitly,
\[
0\coloneqq  \left[\faktor{\Spec(\kappa)}{\mathbb{G}_{m,R}}\right].
\]

\begin{definition}\label{def_S_complete}
We say that $\X \in \mathrm{St}_k$ is $\mathrm{S}$-complete if for any DVR $R$ and any diagram of solid arrows
\[\begin{tikzcd}[sep=small]
	{\ST\setminus \left\lbrace 0 \right\rbrace} && \X \\
	\\
	\ST
	\arrow["f", from=1-1, to=1-3]
	\arrow["j"', hook, from=1-1, to=3-1]
	\arrow["{\exists ! \ \widetilde{f}}"', dashed, from=3-1, to=1-3]
\end{tikzcd}\]
there exists a unique extension $\widetilde{f}$ of $f$ making the above triangle commute.
	\end{definition}

One can try to unify \cref{def_Theta_reductive} and \cref{def_S_complete} under the following:

\begin{definition}[{\cite[Remark 3.51]{AHLH}}]\label{def_Hartogs_principle}
We say that $\X \in \mathrm{St}_k$ satisfies the \textit{Hartogs' principle} if for every $X$ integral local regular noetherian scheme of dimension 2 there exists a unique dotting arrow filling the diagram
\[\begin{tikzcd}[sep=small]
	U && \X \\
	\\
	X
	\arrow["F", from=1-1, to=1-3]
	\arrow["j"', hook, from=1-1, to=3-1]
	\arrow["{j_*F}"', dashed, from=3-1, to=1-3]
\end{tikzcd}\]
where $U$ is complement of the unique closed of $X$ and $j \colon U \to X$ the open immersion.
\end{definition}

\begin{remark}[{\cite[Remark 3.51]{AHLH}}]\label{Hartogs_implies_existence}
An algebraic stack satisfying the Hartogs' principle is automatically $\Theta$-reductive and $\mathrm{S}$-complete, although there are counterexamples for the converse.
\end{remark}

Let us now state the main theorem of \cite{AHLH}:

\begin{theorem}[{\cite[Theorem A]{AHLH}}]\label{existence_good_separated}
Let $S$ be an algebraic space of characteristic $0$.
Let $\X$ be an algebraic stack of finite presentation with affine stabilizers and separated diagonal over $S$.
Then $\X$ admits a separated good moduli space $X$ if and only if $\X$ is $\Theta$-reductive and $\mathrm{S}$-complete.
\end{theorem}

Moduli of flat objects have affine diagonal, as shown by the following:

\begin{lemma}[{\cite[Lemma 7.20]{AHLH}}]\label{affine_diagonal}
Let $\C \in \Prlo_k$ of finite type equipped with be an admissible $t$-structure $\tau$ universally satisfying openness of flatness.
Then $\M_\C^\heartsuit$ has affine diagonal.
\end{lemma}

\begin{proof}
	By \cref{geometric_heart} we have that $\M_\C^\heartsuit$ is an algebraic stack locally of finite presentation.
	Then the proof of \cite[Lemma 7.20]{AHLH} goes through without changes.
\end{proof}

\begin{proposition}[{\cite[Corollary 3.2.11]{good_Perv}}]\label{closed_point_is_semisimple}
	Let $\C \in \Prlo_k$ of finite type equipped with be an admissible $t$-structure $\tau$ universally satisfying openness of flatness. 
	For $\Spec(\kappa) \to \Spec(k)$ a closed point with $\kappa$ algebraically closed, the closed $\kappa$-points of $\M_\C^\heartsuit$ parametrize pseudo-perfect semsimple objects of $\C_\kappa^\heartsuit$.
\end{proposition}

In our setting having affine diagonal (thus affine stabilizers) and being locally of finite presentation come for free (\cref{affine_diagonal} and \cref{geometric_heart}), so that in \cref{existence_good_separated} one only needs to check quasi-compactness, $\Theta$-reductiveness and $\mathrm{S}$-completeness.
The stack $\M_\C^{\left[0,0\right]}$ will essentially never be quasi-compact, but $\Theta$-reductiveness and $\mathrm{S}$-completeness are stable under taking closed substacks.
The next result gives a way to cut quasi-compact substacks.

\begin{notation}\label{nu_substacks}
Let $\C \in \Prlo_k$ be a finite type category.
Let $E \in \C$ be a compact generator and $\nu: \ZZ \to \N$ be a function with finite support.
Define a presheaf
\[
\M_\C^\nu: \dAff_k^{op} \to \Spc
\]
that assigns to $\Spec(A)$ the maximal groupoid
\[
\M_\C^\nu(A) \subset \C_A^{\simeq}
\]
spanned by those $F \in \Funst_k((\C^{\omega})^{op}, \Perf(A))$ for which the following property holds: for every field $K$ and every $A \to \pi_0(A) \to K$ we have $\dim_K \pi_i(F(E)\otimes_A K) \leq \nu(i)$.\\
	This defines a substack of $\M_\C$, which depends on the choice of the compact generator $E$.
\end{notation}

\begin{proposition}[{\cite[Proposition 3.20]{TV}}]\label{quasi_compact_substacks}
In the setting of Notation \ref{nu_substacks}, the derived stack $\M_\C^\nu \subset \M_\C$ defines an open substack of finite presentation over $k$.
\end{proposition}

\begin{warning}
	The $\Theta$-reductiveness and $S$-completeness are preserved under taking closed substacks.
	The substacks of \cref{quasi_compact_substacks} are open, but not closed in general.\\
	Actually we will be interested in the quasi-compact substack stack defined by the pullback square of open immersions
\[\begin{tikzcd}[sep=small]
	{\M_\C^{\left[0,0\right], \nu}} && {\M_\C^{\left[0,0\right]}} \\
	\\
	{\M_\C^\nu} && {\M_\C}
	\arrow[circled, hook, from=1-1, to=1-3]
	\arrow[circled, hook, from=1-1, to=3-1]
	\arrow[circled, hook, from=1-3, to=3-3]
	\arrow[circled, hook, from=3-1, to=3-3]
\end{tikzcd}\]
Since we impose flatness, one often gets that the top horizontal map is both an open and closed immersion (thus a union of connected components).
\end{warning}

Finally, let us recall that a notion of derived good moduli spaces has been introduced in \cite{derived_good}.
The following is the main result of that paper:
\begin{theorem}[{\cite[Theorem 2.12]{derived_good}}]\label{derived_good}
Let $\X$ be a geometric derived stack.
\begin{enumerate}\itemsep=0.2cm
    \item If $t_0\X$ admits a good moduli space $t_0q \colon  t_0\X \to t_0X$, then $\X$ admits a derived good moduli space $q' \colon  \X \to X$ such that $t_0q' \simeq t_0q$.
    \item If $\X$ admits a derived good moduli space $q \colon  \X \to X$, then $t_0q$ is a good moduli space for $t_0X$.
\end{enumerate}
\end{theorem}

\section{Good moduli spaces for $\mathbf{Rep}_k(\I)^\heartsuit$}\label{good_Rep_section}
In this paragraph we show the existence of good moduli spaces for moduli of representations.

\begin{lemma}\label{Hartogs_loc_free}
	Let $X$ be a discrete regular integral noetherian scheme of dimension $2$.
	Let $U \subset X$ be the complement of a $0$-dimensional closed subscheme.
	Then the functor
\[
\mathrm{H}^0(j_\ast) \colon \QCoh(U)^\heartsuit \to \QCoh(X)^\heartsuit
\]
preserves locally free sheaves of finite rank.
\end{lemma}

\begin{proof}
	By \cite[Corollary 1.4]{Hartshorne}, locally free sheaves on $X$ of finite rank agree with reflexive sheaves, that is, quasi-coherent sheaves that are equivalent to their double dual.
	By \cite[Corollary 1.4 \& Proposition 1.6-(3)]{Hartshorne}, the sheaf $\mathrm{H}^0(j_*)F$ is a locally free sheaves of finite rank and is the unique such extension.
\end{proof}

\begin{theorem}\label{Hartogs_Rep}
Let $\I \in \Cat_\infty^\omega$.
The algebraic stack $\mathbf{Rep}_k(\I)^\heartsuit$ satisfies the Hartogs' principle.
\end{theorem}

\begin{proof}
	Let $X = \Spec(A)$ be a discrete normal integral noetherian affine scheme and $j\colon U \hookrightarrow X$ the complement of a $0$-dimensional closed subscheme of $X$.
	We need to show that every diagram of solid arrows
\begin{equation}\label{Hartogs_diagram}
\begin{tikzcd}[sep=small]
	U && {\mathbf{Rep}_k(\I)^\heartsuit} \\
	\\
	X
	\arrow["F", from=1-1, to=1-3]
	\arrow["j"', hook, from=1-1, to=3-1]
	\arrow[dashed, from=3-1, to=1-3]
\end{tikzcd}
\end{equation}
	By descent, a map $U \to \mathbf{Rep}_k(\I)^\heartsuit$ corresponds to representation $F \colon \I \to \QCoh(U)^\heartsuit$ that take values in locally free sheaves of finite rank.
	Consider the representation
\[
\mathrm{H}^0(j_\ast)F \colon \I \to \QCoh(X)^\heartsuit
\]
\[
c \mapsto \mathrm{H}^0(j_\ast )(F(c))
\]
	By \cref{Hartogs_loc_free}, $\mathrm{H}^0(j_\ast )F$ takes values in locally free sheaves of finite rank and it is the unique such extension.
	This shows that $\mathrm{H}^0(j_\ast) F$ is the unique extension of $F$ that makes the triangle \eqref{Hartogs_diagram} commute.
\end{proof}

\begin{definition}
	Let $\I \in \Cat_\infty^\omega$
	Let $A$ be a discrete commutative algebra over $k$ and $F \in \C^\heartsuit_A$ be a pseudo-perfect $\tau_A$-flat object.
	We say that $F$ lies over a substack $\X \subset \REP$ if the associated morphism $x_F \colon \Spec(A) \to \REP$ factors through $\X$.
\end{definition}

Let $\Spec(\kappa) \to \Spec(k)$ be a closed point with $\kappa$ algebraically closed. \cref{Hartogs_Rep} yields the following

\begin{corollary}\label{good_Rep}
	Let $\I \in \Cat_\infty^\omega$.
	The algebraic stack $\REP$ is $\Theta$-reductive and $\mathrm{S}$-complete.
	In particular, any quasi-compact closed substack $\X \subset \REP$ admits a separated good moduli space $X$ whose $\kappa$-points parametrize pseudo-perfect semisimple objects over $\X$.
	Moreover $X$ comes with a natural derived enhancement if $\X$ is the truncation of a derived substack of $\dREP$.
\end{corollary}

\begin{proof}
	By \cref{Hartogs_Rep} and \cref{Hartogs_implies_existence}, $\REP$ is $\Theta$-reductive and $\mathrm{S}$-complete. Moreover $\REP$ has affine diagonal by \cref{affine_diagonal}.
	Since these properties are stable under taking closed substacks, $\X$ satisfies the assumptions of \cref{existence_good_separated}.
	The characterization of $\kappa$-points of $X$ follows from \cref{good_moduli_properties}-(4) and \cref{closed_point_is_semisimple}.
	Derived enhancement are provided by \cref{derived_good}.
\end{proof}

\begin{construction}\label{fixed_rank_Rep}
	Let $\I \in \Cat_\infty^\omega$ and choose a finite set of representative $\left\lbrace c_a \right\rbrace_{a \in I}$ of equivalence classes of $\I$ (\cref{cpt_gen_Rep}).
The proof of \cref{openness_flatness_Rep} shows that there is a pullback square
\[\begin{tikzcd}[sep=small]
	{\dREP} && {\mathbf{Rep}_k(\I)} \\
	\\
	{\prod_{a \in I}\mathbf{Vect}_k} && {\prod_{a \in I}\mathbf{Perf}_k}
	\arrow[from=1-1, to=1-3]
	\arrow[from=1-1, to=3-1]
	\arrow[from=1-3, to=3-3]
	\arrow[from=3-1, to=3-3]
\end{tikzcd}\]
where the right vertical map is given by the product of evaluations at any $c_a$, $a \in I$, and $\mathbf{Vect}_k \coloneqq \bigsqcup_{n \in \N} \mathrm{B}\mathbf{GL}_n$.
Choose a set of non-negative integers 
\[
\underline{r}=\left\lbrace r_a\right\rbrace_{a \in I} \subset \N
\] and define the derived  stack $\mathbf{Rep}_{k}^{\underline{r}}(\mathcal{I})$ via the pullback square
\[\begin{tikzcd}[sep=small]
	{\mathbf{Rep}_{k}^{\underline{r}}(\mathcal{I})} && {\dREP} \\
	\\
	{\prod_{a \in I} \mathrm{B}\mathbf{GL}_{r_a}} && {\prod_{a \in I}\mathbf{Vect}_k}
	\arrow[from=1-1, to=1-3]
	\arrow[from=1-1, to=3-1]
	\arrow[from=1-3, to=3-3]
	\arrow[from=3-1, to=3-3]
\end{tikzcd}\]
The stack $\mathbf{Rep}_{k}^{\underline{r}}(\mathcal{I})$ does not depend on the choice of the $c_a$'s, but only on the choice of the set of integers $\underline{r}$.
\end{construction}

\begin{proposition}\label{quasi_compact_Rep}
In the setting of \cref{fixed_rank_Rep}, the derived stack $\mathbf{Rep}_{k}^{\underline{r}}(\mathcal{I})$ is a quasi-compact open and closed substack of $\dREP$.
In particular, it is a $1$-Artin stack of finite presentation over $k$.
\end{proposition}

\begin{proof}
Since $I$ is finite, it follows that the map $\prod_{a \in I} \mathrm{B}\mathbf{GL}_{r_a} \to \prod_{a \in I}\mathbf{Vect}_k$ is open and closed and quasi-compact.
It then follows that $\mathbf{Rep}_{k}^{\underline{r}}(\mathcal{I}) \to \dREP$ is also open and closed and quasi-compact.
This proves that $\mathbf{Rep}_{k}^{\underline{r}}(\I)$ is a $1$-Artin stack locally of finite presentation. We are left to prove that it is also quasi-compact.
Consider the function
\[
\nu: \ZZ \to \N
\]
\begin{align*}
m \mapsto 
\left\{
    \begin{aligned}
        & \sum_{a \in I} r_a \ \ & \text{if $n=0$,} \\
        & \ \ \ \ 0 & \text{if $ n\neq 0$.}                  
    \end{aligned}
\right.
\end{align*}
	By \cref{cpt_gen_Rep}, there exists a compact generator $E$ such that the evaluation at $E$ of Notation \ref{nu_substacks} is given by the direct sums of evaluation at the elements $c_a$, $a \in I$.
	It follows that the quasi-compact stack $\mathbf{Rep}_{k}^{\nu}(\I)$ of \cref{quasi_compact_substacks} is such that we have an open immersion $\mathbf{Rep}_{k}^{\underline{r}}(\mathcal{I}) \subset \mathbf{Rep}_{k}^{\nu}(\I)$.
	Hence we have a diagram with pullback squares
\[\begin{tikzcd}[sep=small]
	{\mathbf{Rep}_{k}^{\underline{r}}(\I)} && {\mathbf{Rep}_{k}^{\nu}(\I)} \\
	\\
	{\mathbf{Rep}_{k}^{\underline{r}}(\I)} && {\mathbf{Rep}_{k}(\I)} \\
	\\
	{\prod_{a \in I} \mathrm{B}\mathbf{GL}_{r_a}} && {\prod_{a \in I} \mathbf{Perf}_k}
	\arrow[from=1-1, to=1-3]
	\arrow["id"', from=1-1, to=3-1]
	\arrow[from=1-3, to=3-3]
	\arrow[from=3-1, to=3-3]
	\arrow[from=3-1, to=5-1]
	\arrow[from=3-3, to=5-3]
	\arrow[from=5-1, to=5-3]
\end{tikzcd}\]
where the horizontal maps are quasi-compact.
In particular, $\mathbf{Rep}_{k}^{\underline{r}}(\mathcal{I})$ is quasi-compact.
\end{proof}

\begin{corollary}\label{good_Rep_rank}
In the setting of Construction \ref{fixed_rank_Stokes}, the algebraic stack $t_0 \mathbf{Rep}_{k}^{\underline{r}}(\I)$ admits a separated good moduli space $t_0 \Rep_{k}^{\underline{r}}(\I)$.
Moreover $t_0 \Rep_{k}^{\underline{r}}(\I)$ admits a derived enhancement $\Rep_{k}^{\underline{r}}(\I)$ which is a derived good moduli space for $\mathbf{Rep}_{k}^{\underline{r}}(\I)$.
\end{corollary}

\begin{proof}
It follows directly from \cref{quasi_compact_Rep} that we can apply \cref{good_Rep}.
\end{proof}

	Let $\Spec(\kappa) \to \Spec(k)$ be a closed point with $\kappa$ algebraically closed.
	We have the following 
\begin{theorem}\label{good_Rep_theorem}
	Let $\I \in \Cat_\infty^\omega$.
	The algebraic stack $\REP$ admits a good moduli space $\underline{\Rep}_k(\I)^\heartsuit$ whose $\kappa$-points parametrize semisimple representations of finite-dimensional $\kappa$-vector spaces.
	Moreover $\underline{\Rep}_k(\I)^\heartsuit$ admits a derived enhancement $\underline{\Rep}_k(\I)^{[0,0]}$ which is a derived good moduli space for $\mathbf{Rep}_{k}(\I)^{[0,0]}$.
\end{theorem}

\begin{proof}
	In the setting of \cref{good_Rep_rank}, the morphism
\[
q \colon \mathbf{Rep}_{k}(\I) \simeq \bigsqcup_{\underline{r} \in \N^I} \mathbf{Rep}_{k}^{\underline{r}}(\I) \to \bigsqcup_{\underline{r} \in \N^I}\Rep_{k}^{\underline{r}}(\I)
\]
is a derived good moduli space.
	The result then follows by combining \cref{derived_good} with \cref{good_moduli_properties}-(4) and \cref{closed_point_is_semisimple}.
\end{proof}

\subsection{Application: constructible sheaves}
In this paragraph we recall the exodromy equivalence as appearing in \cite{Exodromy, exodromyconicality} and the construction of the stack of constructible sheaves.

\begin{notation}
	Throughout this section we will work with hypersheaves instead of sheaves.
	In the application we are interested in, there is no difference between the two notions.
	Since there is no risk of confusion, we will not make any reference to this in our notation.
\end{notation}

\begin{definition}
	A stratified space is the data of a continuous map $\rho \colon X \to P$ where $X$ is a topological space and $P$ is a poset endowed with the Alexandroff topology.
	Abusing notations, we will often refer to a stratified space as a pair $(X,P)$.
	For $p \in P$, the $p$-stratum of $X$ is $X_p \coloneqq \rho^{-1}(p) \subset X$.
\end{definition}

\begin{definition}
	Let $(X,P)$ be a stratified space and let $\E \in \Cat_\infty$. 
	For $p \in P$, denote by $i_p \colon X_p \to X$ the inclusion.
	An object $F \in \Sh(X; \E)$ is constructible if $i_p^\ast F$ is locally constant for every $p \in P$.
	We denote by
	\[
	\Cons_P(X; \E) \subset \Sh(X ; \E)
	\]
	the full subcategory spanned by constructible sheaves.
\end{definition}

\begin{definition}
Let $\D \in \Cat_\infty$.
An object $d \in \D$ is atomic if the functor
\[
\Map_{\D} (d, -) \colon \D \to \mathrm{Spc}
\]
commutes with colimits. We denote by $\D^{at}\subset \D$ the full subcategory spanned by atomic objects.
\end{definition}

\begin{definition}
An $\infty$-category $\D \in \Cat_\infty$ is atomically generated if the unique colimit preserving functor
$\Fun((\D^{at})^{op}, \mathrm{Spc}) \to \D$
extending $\D^{at}\subset \D$ is an equivalence.
\end{definition}

\begin{definition}[{\cite[Definition 3.5]{CJ}}]\label{def_exodromic}
A stratified space $\rho \colon X \to P$ is exodromic if the following conditions ares satisfied:
\begin{enumerate}\itemsep=0.2cm
	\item the $\infty$-category $\Cons_P(X; \mathrm{Spc})$ is atomically generated;
	\item the subcategory $\Cons_P(X; \mathrm{Spc}) \subset \Sh(X)$ is closed under both limits and colimits;
	\item the pullback functor $\rho^{\ast}\colon \Fun(P,\mathrm{Spc}) \to \Cons^{\hyp}_P(X)$ preserves limits.
\end{enumerate}
\end{definition}

\begin{example}\label{exodromic_example}
	\cref{def_exodromic} enjoy many stability properties (see \cite[Theorem 5.1.7]{exodromyconicality}).
	In particular, any stratified space locally admitting a conically stratified refinement with locally weakly contractible strata is exodromic.
	Example of such are real algebraic varieties with finite stratification by Zariski locally closed subsets, and subanalytic spaces with locally finite stratification by subanalytic subset (\cite[Theorem 5.3.9 \& Theorem 5.3.13]{exodromyconicality}).
\end{example}

\begin{remark}[Exodromy equivalence]\label{exodromy_formal}
It follows directly from \cref{def_exodromic} that for
\[
\Pi_\infty(X,P) \coloneqq (\Cons_P(X; \mathrm{Spc})^{at})^{op}
\]
there is an equivalence
\[
\Fun(\Pi_\infty(X,P), \mathrm{Spc}) \simeq \Cons_P(X; \mathrm{Spc})
\]
Moreover \cite[Corollary 4.1.15]{exodromyconicality} shows that this equivalence extends to any compactly assembled $\infty$-category $\E$, that is: there is an equivalence
\[
\Fun(\Pi_\infty(X,P), \E) \simeq \Cons_P(X; \E)
\]
\end{remark}

\begin{remark}
The exodromy equivalence of \cref{exodromy_formal} is a formal statement that follows directly from the definitions.
The exodromy equivalence of \cite{Exodromy} shows that the $\infty$-category $\Pi_\infty(X,P)$ is identified with Lurie's simplicial model for exit path's when $(X,P)$ is a conically stratified space (\cite[Definition A.6.2]{HA}) with locally weakly contractible strata.
See \cite[Subsection 5.6]{exodromyconicality} for more details.
\end{remark}

\begin{definition}[{\cite[Definition 2.2.1]{Exodromy}}]\label{finite_strat}
Let $\left(X,P\right)$ be a conically stratified space.
\begin{itemize}\itemsep=0.2cm
    \item[(i)] We say that a stratified space $\left(X,P\right)$ is categorically compact if $\Exit\left(X,P\right)$ is a compact object in $\Cat_\infty$.
    \item[(ii)] We say that $\left(X,P\right)$ is locally categorically compact if $X$ admits a fundamental system of open subsets $U$ such that $\left(U,P\right)$ is categorically compact.
\end{itemize}
\end{definition}

\begin{example}[{\cref{exodromic_example}} continued]
Real algebraic varieties with finite stratification by Zariski locally closed subsets are exodromic, categorically compact and locally categorically compact (\cite[Theorem 5.3.13]{exodromyconicality}).
\end{example}

\begin{example}
Compact subanalytic spaces with locally finite stratification by subanalytic subset are exodromic, categorically compact and locally categorically compact (\cite[Theorem 5.3.9]{exodromyconicality}).
\end{example}

\begin{proposition}
Let $(X, P)$ be an exodromic stratified space.
If $(X,P)$ is categorically compact, then $\Cons_P(X; \Mod_k)\in \Prlo_k$ is of finite type. 
\end{proposition}

\begin{proof}
This is \cite[Lemma 7.1.9-(3)]{Exodromy} combined with \cref{exodromy_formal}.
\end{proof}

\begin{recollection}
Let $(X,P)$ be a stratified space. Consider the prestack
\[
\mathbf{Cons}_P(X) \colon \dAff_k^{op} \to \mathrm{Spc}
\]
\[
\Spec(A) \mapsto \Cons_{P, \omega}(X; \Mod_A)^{\simeq} \coloneqq \Fun(\Pi_\infty(X,P), \Perf(A))^{\simeq}
\]
sending a morphism $\Spec(A) \to \Spec(B)$ over
$k$ to the map on maximal sub-$\infty$-groupoids induced by
\[
(-) \otimes_B A \colon \Cons_{P,\omega}(X; \Mod_B) \to \Cons_{P,\omega}(X; \Mod_A).
\]
Since $\Mod(-)$ and $\Perf(-)$ satisfies hyper-descent (\cite[Proposition 2.8.4.2-(10) \& Corollary D.6.3.3]{SAG}), the above assignment defines a stack.
\end{recollection}

\begin{theorem}\label{Cons_Rep}
Let $(X, P)$ be an exodromic stratified space. If
$(X,P)$ is categorically compact, then we have an equivalence in $\dSt_k$ 
\[\mathbf{Cons}_P(X) \simeq \M_{\Cons_P(X; \Mod_k)} = \mathbf{Rep}_k(\Pi_\infty(X,P)) \ .
\]
In particular, $\mathbf{Cons}_P(X)$ is a locally geometric derived stack, locally of finite presentation.
\end{theorem}

\begin{proof}
This is (the proof of) \cite[Theorem 5.4.9]{exodromyconicality}. See also (the proof of) \cite[Theorem 7.1.12]{Exodromy}.
\end{proof}

\begin{notation}\label{notation_Cons}
	Let $(X, P)$ be an exodromic categorically compact stratified space
	We denote by
\[
\mathbf{Cons}_P^{[0,0]}(X) \subset \mathbf{Cons}_P(X)
\]
the open substack corresponding to 
\[
\mathbf{Rep}^{[0,0]}_k(\Pi_\infty(X,P)) \subset \mathbf{Rep}_k(\Pi_\infty(X,P))
\]
under the equivalence of \cref{Cons_Rep}. As customary, we denote by $\mathbf{Cons}_P^\heartsuit (X)$ the classical truncation of $\mathbf{Cons}_P^{[0,0]}(X)$.
\end{notation}

Let $\Spec(\kappa) \to \Spec(k)$ be a closed point with $\kappa$ algebraically closed.
We have the following

\begin{theorem}\label{good_rank_Cons}
	In the setting of \cref{notation_Cons}, the algebraic stack $\mathbf{Cons}_P^{\heartsuit}(X)$ admits a separated good moduli space $\Cons_P^{\heartsuit}(X)$ whose $\kappa$-points parametrize semisimple objects of $\Cons_P(X; \Mod_\kappa^\heartsuit)$ with finite-dimensional stalks.
	Moreover $\Cons_P^{\heartsuit}(X)$ admits a derived enhancement $\Cons_P^{[0,0]}(X)$ which is a derived good moduli space for $\mathbf{Cons}_P^{[0,0]}(X)$.
\end{theorem}

\begin{proof}
	This follows directly from \cref{good_Rep_theorem}.
\end{proof}

\section{Good moduli spaces for Stokes functors}\label{Stokes_section}
In this paragraph we prove that the main results of \cref{good_Rep_section} also holds for the moduli of Stokes functors. Since Stokes functors are defined as a full subcategory of a category of representations, the results in \cref{good_Rep_section} cannot be applied directly, although very little changes are needed.

\subsection{Stokes functors}
We recollect in this paragraph the definition and some results about Stokes functors as treated in \cite{Abstract_Stokes, Geometric_Stokes}.
In these two seminal papers several equivalent definitions of Stokes functors are presented.
We have chosen the one that is more explicit and for which it is easier to describe the moduli space we are interested in.
Stokes functors will be defined as a full subcategory of a category of functors, satisfying two properties: Stokes functors are \textit{punctually split} (\cref{def_PS}) and \textit{cocartesian} (\cref{def_cocart}).

\begin{recollection}\label{recollection_PS}
Let $\I \to \X$ be a cocartesian fibration in posets.
Consider the underlying fibration in Sets, denoted $\I^{set} \to \X$ and informally described by forgetting the order in each fiber.
There is a natural functor $i_{\I}\colon \I^{set} \to \I$ above $\X$.\\
Consider a cocomplete $\infty$-category $\E$.
For $x \in \X$ we can consider the inclusion $j_x\colon \I_x \hookrightarrow \I$ of the fiber above $x$ and the morphism $i_{\I_x}\colon \I_x^{set} \to \I_x$.
The restriction functor
\[
i_{\I_x}^\ast  \colon \Fun(\I_x, \E) \to \Fun(\I_x^{set}, \E)
\]
admits a left adjoint
\[
i_{\I_x,!}\colon \Fun(\I_x^{set}, \E) \to \Fun(\I_x, \E)
\]
defined by left Kan extension. Similarly, there is an adjunction $i_{\I,!} \dashv i_{\I}^\ast$.
\end{recollection}

\begin{definition}\label{def_PS}
In the setting of Recollection \ref{recollection_PS}, we say that $F \in \Fun(\I, \E)$ is:
\begin{enumerate}\itemsep=0.2cm
    \item punctually split at $x$ if $j_x^\ast F$ belongs to the essential image of $i_{\I_x,!}\colon \Fun(\I_x^{set}, \E) \to \Fun(\I_x, \E)$;
    \item punctually split if it is so at any point $x$ of $\X$;
    \item split if it belongs to the essential image of $i_{\I,!}\colon \Fun(\I^{set}, \E) \to \Fun(\I, \E)$
\end{enumerate}
\end{definition}

The cocartesian condition is a property about "parallel transport" along a path.
\begin{recollection}\label{recollection_cocart}
Let $\I \to \X$ be a cocartesian fibration in posets.
Consider a morphism $\gamma\colon x \to y$ in $\X$. We can the choose (uniquely up to homotopy) a cocartesian lift $\I_x \to \I_y$, that we still (abusively) denote $\gamma$.
This morphism fits in a \textit{non commutative} diagram
\[\begin{tikzcd}[sep=small]
	{\I_x} && {\I_y} \\
	\\
	& \I
	\arrow["\gamma", from=1-1, to=1-3]
	\arrow["{j_x}"', hook, from=1-1, to=3-2]
	\arrow["{j_y}", hook', from=1-3, to=3-2]
\end{tikzcd}\]
which comes equipped with a natural transformation $s \colon j_x \to j_y \circ \gamma$.\\
Consider a cocomplete $\infty$-category $\E$.
Then applying $\Fun(-,\E)$ to $s$ gives a natural transformation $s^\ast \colon j_x^\ast \to \gamma^\ast j_y^\ast$.
Passing to the induced Beck-Chevalley transformation we get a natural transformation
\[
\eta_\gamma \colon \gamma_! j_x^\ast \to j_y^\ast.
\]
\end{recollection}

\begin{definition}\label{def_cocart}
	In the setting of Recollection \ref{recollection_cocart}, we say that $F \in \Fun(\I, \E)$ is cocartesian at $\gamma \colon x \to y$ if the natural transformation $\eta_\gamma F  \colon \gamma_! j_x^\ast F \to j_y^\ast F$ is an equivalence.
	We say that $F$ is cocartesian if it is so at any morphism $\gamma \colon x \to y$ of $\X$.
\end{definition}

	We can now state the definition of Stokes functors:
\begin{definition}
	Let $\I \to \X$ be a cocartesian fibration in posets and $\E$ be a cocomplete $\infty$-category.
	The $\infty$-category of Stokes functors is the full subcategory 
\[
\St_{\I, \E} \subset\Fun(\I, \E)
\]
spanned by punctually split cocartesian functors.
\end{definition}

\begin{notation}
	If $A$ is a simplicial commutative ring, we will use the shortcut $\St_{\I, A}$ for $\St_{\I, \Mod_A}$.
\end{notation}

\begin{remark}\label{t_structure_Stokes}
Let $\E$ be a stable presentable $\infty$-category equipped with a $t$-structure $\tau$.
It is shown in \cite[Proposition 5.7.11]{Abstract_Stokes} that if $\mathrm{St}_{\I, \E} \subset\Fun(\I, \E)$ is closed under limits and colimits, the the standard $t$-structure $\taust$ on $\Fun(\I, \E)$ restricts to a $t$-structure on $\mathrm{St}_{\I, \E}$.
The restricted $t$-structure on $\mathrm{St}_{\I, \E}$ will also be denoted by $\taust$.
\end{remark}

\begin{proposition}\label{push_Stokes_heart}
Let $f \colon \E \to \E'$ be in $\Pr^{\mathrm{L}}$.
Let $\I \to \X$ be a cocartesian fibration in finite posets.
Suppose that $\E$ and $\E'$ are equipped with $t$-structures and that $f$ preserve coconnective objects.
Assume that $\mathrm{St}_{\I, \E} \subset\Fun(\I, \E)$ and $\mathrm{St}_{\I, \E} \subset\Fun(\I, \E)$ are closed under limits and colimits.
The functor
\[
\tau_{\geq 0} \circ f \circ -  \colon \Fun(\I,\E'^\heartsuit) \to \Fun(\I, \E^\heartsuit)
\]
preserves Stokes functors.
\end{proposition}

\begin{proof}
In \cite[Proposition 5.6.1]{Abstract_Stokes} it is shown that Stokes functors are stable under change of coefficients in $\Pr^{\mathrm{L}}$.
Combining this with \cite[Proposition 5.7.11]{Abstract_Stokes}, we see that
\[
\tau_{\geq 0} \circ f \circ -  \colon \Fun(\I, \E) \to \Fun(\I, \E'_{\geq 0})
\]
preserves Stokes functors.
Since by assumption $f$ respect coconnective objects, the above functor restricts to a functor
\[
\tau_{\geq 0} \circ f \circ -  \colon \Fun(\I,\E^\heartsuit) \to \Fun(\I, \E'^\heartsuit)
\]
that preserves Stokes functors.
\end{proof}

\begin{recollection}[{\cite[Notation 8.1.1]{Geometric_Stokes}}]
Let $\I \to \X$ be a cocartesian fibration in posets.
We can define a prestack via the following rule:
\[
\mathbf{St}_{\I, k} \colon \dAff_k \to \Spc
\]
\[
\Spec(A) \mapsto (\mathrm{St}_{\I, \Mod_A, \omega})^\simeq,
\]
where $(\mathrm{St}_{\I, A, \omega})^\simeq$ is the maximal $\infty$-groupoid of 
\[
\mathrm{St}_{\I, A, \omega} \coloneqq \mathrm{St}_{\I, A} \times_{\Fun(\I, \Mod_A)}\Fun(\I, \Perf(A)).
\]
Under suitable assumptions, \cite[Proposition 8.2.2]{Geometric_Stokes} shows that $\mathbf{St}_{\I, k} \simeq \M_{\\St_{\I, k}}$.
\end{recollection}

We will not recall here what the most general framework in which $\mathbf{St}_{\I, k} \simeq \M_{\\St_{\I, k}}$ holds true, but instead refer to \cite[Theorem 8.1.3]{Geometric_Stokes}.
We will instead specialize to the particular case of classical (ramified) Stokes functors, but the same arguments presented here work in the very general setting of \cite[Theorem 8.1.3]{Geometric_Stokes}).

\subsubsection{Sheaves of unramified irregular values}\label{unramified_section}

\begin{notation}\label{notation_divisors}
Let $X$ be a complex manifold and let $D \subset X$ be a strict normal crossing divisor in $X$.
Write $D$ as the union of its irreducible components
\[
D = \bigcup_{i=1}^n D_i
\]
We denote (abusively) by $(X,D)$ the topological stratified space with stratification given by depth of the intersection. That is, the stratified space
\[
\rho \colon X \to \Fun(\left\lbrace 1, \ldots n\right\rbrace), \Delta^1)
\]
\[
x \mapsto [i \to \delta_{x,i}]
\]
where $\delta_{x,i}= 1$ if $x \in D_i$ and $\delta_{x,i}= 0$ if $x \notin D_i$.
We will refer to $(X,D)$ as being a strict normal crossing pair.
\end{notation}

\begin{notation}
Let $(X,D)$ be a strict normal crossing pair.
For $I\subset\left\lbrace 1, \ldots n\right\rbrace)$, we put
\[
D_I  \coloneqq  \bigcap_{i\in I} D_i  \quad \text{ and } \quad  D_I^{\circ}  \coloneqq  \bigcap_{I\subsetneq J} D_I \setminus D_J   \ .
\]
We denote by $i_I \colon  D_I \hookrightarrow  X$ and $i_I^{\circ} \colon  D_I \hookrightarrow X$ the canonical inclusions.
\end{notation}

\begin{construction}[{\cite[Section 8.b]{Sabbah}}]\label{real_blow_up}
Let $(X,D)$ be a strict normal crossing pair.
Assume that $X$ admits a smooth compacitfication.
For $i=1, \ldots, n$, let $L(D_i)$ be the line bundle over $X$ corresponding to the sheaf $\OO_X(D_i)$ and let $S^1 L(D_i)$ be the associated circle bundle. 
Put
\[
S^1 L(D)  \coloneqq  \bigoplus_{i=1}^n S^1 L(D_i)   \ .
\]
Let $U\subset X$ be an open polydisc with coordinates $(z_1,\dots, z_n)$ and let $z_i=0$ be an equation of $D_i$ in $U$.
Let $\widetilde{X}_U \subset S^1 L(D)|_U$ be the closure of the image of $(z_i/|z_i|)_{1\leq i \leq n} \colon  U\setminus D \to S^1 L(D)$.
Then, the $\widetilde{X}_U$ are independent of the choices made and thus glue as a closed subspace $\widetilde{X} \subset S^1 L(D)$ called the \textit{real-blow up of $X$ along $D$}.
We denote by $\pi \colon  \widetilde{X}\to X$  the induced proper morphism and by $j \colon  X\setminus D \to \widetilde{X}$ the canonical open immersion.
For $I\subset \{1,\dots, n\}$ of cardinal $1 \leq k \leq n$, we put $\widetilde{D}_I \coloneqq  \pi^{-1}(D_I )$ and $\widetilde{D}_I^{\circ} \coloneqq  \pi^{-1}(D_I^{\circ} )$
and observe that the restriction 
\[
\pi|_{D_I^{\circ} }  \colon   \widetilde{D}_I^{\circ} \to D_I^{\circ}
\]
 is a $S^k$-bundle.
\end{construction}

\begin{recollection}[{\cite[Section 8.c]{Sabbah}}]\label{moderate_growth_sheaf}
Let $(X,D)$ be a strict normal crossing pair and put $U \coloneqq X\setminus D$.
Let $\pi \colon  \widetilde{X}\to X$  be the real blow-up along $D$ and let $j \colon U \hookrightarrow \widetilde{X}$ be the canonical inclusion.
We denote by $\AA^{\mathrm{mod}}_{\widetilde{X}} \subset  j_{\ast}\OO_{U}$ the sheaf of analytic functions with moderate growth along $D$.
By definition for every open subset $V \subset \widetilde{X}$, a section of $\AA^{\mathrm{mod}}_{\widetilde{X}}$ on $V$ is an analytic function $f \colon  V\cap U \to \mathbb{C}$ such that for every open subset $W \subset V$ with $D$ defined by $h=0$ in a neighborhood of $\pi(W)$, for every compact subset $K \subset W$, there exist  $C_K> 0$ and $N_K \in \mathbb{N}$ such that for every $z\in K\cap U$, we have 
\[
|f(z)|\leq C_K \cdot |h(z)|^{-N_K} \ .
\]
Let $(j_\ast \OO_U)^{\mathrm{lb}} \subset j_\ast \OO_U$ the subheaf of locally bounded function. Then $\AA^{\mathrm{mod}}_{\widetilde{X}}$ is a unitary sub-$(j_\ast \OO_U)^{\mathrm{lb}}$-algebra and $\AA^{\mathrm{mod}, \times}_{\widetilde{X}} \subset (j_\ast \OO_U)^{\mathrm{lb}}$ (\cite[Lemma 10.1.10]{Geometric_Stokes}).
\end{recollection}

\begin{recollection}[{\cite[Recollection 10.1.11]{Geometric_Stokes} \& \cite[Definition 9.2]{Sabbah}}]\label{order_general}
Let $(X,D)$ be a strict normal crossing pair and put $U \coloneqq X\setminus D$.
Let $\pi \colon  \widetilde{X}\to X$  be the real blow-up along $D$ and let $j \colon U \hookrightarrow \widetilde{X}$ be the canonical inclusion.
For $f,g\in j_{\ast}\OO_{U}$, we write 
\[
f\leq g \text{ if and only if } e^{f-g}\in \AA^{\mathrm{mod}}_{\widetilde{X}}  \ .
\]
By \cref{moderate_growth_sheaf}, the relation $\leq$ induces an order on $(j_{\ast}\OO_{U})/(j_{\ast}\OO_{U})^{\mathrm{lb}}$.
From now on, we view  $(j_{\ast}\OO_{U})/(j_{\ast}\OO_{U})^{\mathrm{lb}}$ as an object of 
$ \Sh^{\hyp}(\widetilde{X},\mathrm{Poset})$.
\end{recollection}

\begin{remark}[{\cite[Remark 10.1.11]{Geometric_Stokes}}]\label{meromorphic_order}
Viewing $\pi^{\ast}\OO_{X}(\ast D)$ inside $j_{\ast}\OO_{U}$, we have 
\[\pi^{\ast}\OO_{X}(\ast D)\cap  (j_{\ast}\OO_{U})^{\mathrm{lb}} = \pi^{\ast}\OO_{X}  \ .
\]
Hence, $\pi^{\ast}(\OO_{X}(\ast D)/\OO_{X})$ can be seen as a subsheaf of $(j_{\ast}\OO_{U})/(j_{\ast}\OO_{U})^{\mathrm{lb}}$.
From now on, we view it as an object of $ \Sh^{\hyp}(\widetilde{X},\mathrm{Poset})$.
\end{remark}

\begin{definition}
Let $X$ be a topological space.
Let $F \in \Sh^{\hyp}(X,\Cat_{\infty})$.
We say that $F $ is \textit{locally generated} if there is a cover by open subsets $U\subset X$ such that for every $x\in U$, the functor $F(U)\to F_x$ is essentially surjective.
\end{definition}

\begin{definition}[{\cite[Definition 10.3.1]{Geometric_Stokes}}]
Let $X\subset \mathbb{C}^k$ be a polydisc with coordinates $(z,y)\coloneqq (z_1,\dots, z_n,y_1,\dots y_{k-n})$.
Let $D$ be the divisor defined by $z_1 \cdots z_n = 0$.
Let $b \in \OO_{X,0}(\ast D)/\OO_{X,0}$ and consider the Laurent expansion 
\[
\sum_{m\in \ZZ^n} b_m(y) z^m \ .
\]
We say that $a$ admits an order if  the set 
\[
\{m \in \ZZ^n \text{ with } b_m \neq 0\}\cup \{0\}
\]
admits a smallest element, denoted by $\mathrm{ord} (b)$.
\end{definition}

\begin{remark}[{\cite[Remark 10.3.2]{Geometric_Stokes}}]
The existence of an order does not depend on a choice of coordinates on $X$.
\end{remark}

\begin{recollection}[{\cite[Definition 2.1.2]{Mochizuki}}]\label{goodness}
Let $(X,D)$ be a strict normal crossing pair.
Let $x\in X$.
A subset $I\subset \OO_{X,x}(\ast D)/\OO_{X,x}$ is \textit{good} if 
\begin{enumerate}\itemsep=0.2cm
\item  every non zero $b \in I$ admits an order with $b_{\mathrm{ord} (b)}$ invertible in  $\OO_{X,x}$.
\item For every distinct $a,b\in I, a-b$  admits an order with $(a-b)_{\mathrm{ord} (a-b)}$ invertible in  $\OO_{X,x}$.
\item The set $\{\mathrm{ord}(a-b), a,b\in I\}\subset  \mathbb{Z}^n$ is totally ordered.
\end{enumerate}
\end{recollection}

\begin{recollection}[{\cite[Definition 2.4.2]{Mochizuki}}]\label{good_sheaf}
Let $(X,D)$ be a strict normal crossing pair.
A \textit{good sheaf of unramified irregular values} is a sheaf of unramified irregular values such that for every $x\in X$,  the set $\mathscr{I}_x \subset \OO_{X,x}(\ast D)/\OO_{X,x}$ is good.
\end{recollection}

\begin{proposition}[{\cite[Corollary 10.3.12]{Geometric_Stokes}}]\label{good_sheaf_constructible_finite_stratification}
Let $(X,D)$ be a strict normal crossing pair.
Suppose that $X$ admits a smooth compactification.
Let $\mathscr{I}\subset \OO_{X}(\ast D)/\OO_{X}$ be a good sheaf of unramified irregular values.
Let $\pi \colon  \widetilde{X}\to X$  be the real blow-up along $D$.
Then, there exists a finite subanalytic stratification 
$\widetilde{X}\to P$ refining $(\widetilde{X},\widetilde{D})$ such that $\pi^{\ast}\mathscr{I}\in \Cons_P(\widetilde{X},\mathrm{Poset})$.
\end{proposition}

\subsubsection{Sheaves of (ramified) irregular values}

We now introduce a generalization of the theory introduced in \cref{unramified_section}, from which we keep the notation.
The idea is that a sheaf of ramified irregular values is a sheaf that becomes unramified after pulling back along a ramified cover of $X$.
We begin by recording the following result.
It shows that the theory of unramified sheaves of irregular values can be presented purely in terms the real blow-up $\widetilde{X}$ of $X$.

\begin{lemma}[{\cite[Lemma 10.6.1]{Geometric_Stokes}}]\label{transfer_irregular_value_sheaf}
Let $(X,D)$ be a strict normal crossing pair.
Let $\pi \colon \widetilde{X}\to X$ be the real blow-up along $D$.
Let $\mathscr{I} \subset \pi^{\ast}(\OO_{X}(\ast D)/\OO_{X})$ be a sheaf.
Then, the following are equivalent:
\begin{enumerate}\itemsep=0.2cm
\item There is a sheaf of unramified irregular values $\mathscr{J}\subset \OO_{X}(\ast D)/\OO_{X}$ such that $\mathscr{I}\simeq \pi^{\ast}\mathscr{J}$.
\item The direct image $\pi_\ast \mathscr{I}\subset \OO_{X}(\ast D)/\OO_{X}$ is a sheaf of unramified irregular values and the counit transformation $\pi^{\ast} \pi_\ast  \mathscr{I}\to \mathscr{I}$ is an equivalence.
\end{enumerate}
\end{lemma}

\begin{definition}[{\cite[Definition 10.6.2]{Geometric_Stokes}}]\label{goodness_upper_unramified_case}
If the equivalent conditions of \cref{transfer_irregular_value_sheaf} are satisfied, we say that $\mathscr{I} \subset  \pi^{\ast}(\OO_{X}(\ast D)/\OO_{X})$ is a sheaf of unramified irregular values.
If furthermore $\pi_\ast  \mathscr{I}$ is a good sheaf of unramified irregular values, we say that $\mathscr{I}$  is a good sheaf of unramified irregular values.
\end{definition} 

\begin{construction}[{\cite[Construction 10.6.4]{Geometric_Stokes} \& \cite[Section 9.c]{Sabbah}}]
\label{remified_irregular_values}
Let $X\subset \mathbb{C}^n$ be a polydisc with coordinates 
$(z_1,\dots, z_n)$.
Let $D$ be the divisor defined by $z_1 \cdots z_l = 0$ and put $U\coloneqq X\setminus D$.
Let $\pi \colon  \widetilde{X}\to X$ be the real blow-up along $D$.
Let $j \colon U \hookrightarrow \widetilde{X}$ be the canonical inclusion.
Define $\rho \colon X_d\to X$ by 
$(z_1,\dots, z_n)\to (z_1^d,\dots, z_l^d, z_{l+1}  , \dots,  z_n)$ for $d\geq 1$ and consider the (not cartesian for $d>1$) commutative square 
\[
\begin{tikzcd}
	\widetilde{X}_d\arrow{r}{\widetilde{\rho}} \arrow{d}{\pi_d} & \widetilde{X} \arrow{d}{\pi} \\
	X_d    \arrow{r}{\rho} & X   
\end{tikzcd} 
\]
of real blow-up along $D$. 
The unit transformation $\OO_{U} \hookrightarrow \rho_{\ast}   \OO_{U_d}$ yields an inclusion 
\[
j_{\ast}\OO_{U} \hookrightarrow j_{\ast}\rho_{\ast}   \OO_{U_d} \ .
\]
On the other hand, the unit transformation  $\pi^{\ast}_d\OO_{X_d}(\ast D) \hookrightarrow  j_{d,\ast}\OO_{U_d}$ yields 
\[
\widetilde{\rho}_{\ast}\pi^{\ast}_d\OO_{X_d}(\ast D) \hookrightarrow  \widetilde{\rho}_{\ast} j_{d,\ast}\OO_{U_d} = j_\ast  \rho_{\ast}\OO_{U_d} \ .
\]
Put
\[
IV_d \coloneqq j_{\ast}\OO_{U} \cap\widetilde{\rho}_{\ast}\pi^{\ast}_d\OO_{X_d}(\ast D)  \subset j_{\ast}\OO_{U}  \ .
\]
As in \cref{meromorphic_order}, we have 
\[
IV_d \cap (j_{\ast}\OO_{U})^{\mathrm{lb}}  = j_{\ast}\OO_{U} \cap \widetilde{\rho}_{\ast}\pi^{\ast}_d\OO_{X_d} \ .
\]
We put
\[
\mathscr{IV}_d \coloneqq IV_d/(IV_d\cap  (j_{\ast}\OO_{U})^{\mathrm{lb}}) \subset (j_{\ast}\OO_{U})/ (j_{\ast}\OO_{U})^{\mathrm{lb}} \ .
\]
For an arbitrary strict normal crossing pair $(X,D)$, the  $\mathscr{IV}_d$, $dn\geq 1$ are defined locally and glue into subshseaves 
\[
\mathscr{IV}_d (X,D) \subset (j_{\ast}\OO_{U})/ (j_{\ast}\OO_{U})^{\mathrm{lb}}  
\] 
for $d\geq 1$.
By \cref{order_general}, we view  $\mathscr{IV}_d(X,D)$ as an object of  $\Sh^{\hyp}(\widetilde{X},\mathrm{Poset})$.
\end{construction}

\begin{definition}[{\cite[Definition 10.6.6]{Geometric_Stokes}}]\label{Kummer_cover}
Let $(X,D)$ be a strict normal crossing pair.
Let $d \geq 1$ be an integer.
A $d$-Kummer cover of $(X,D)$ is an holomorphic map $\rho \colon X\to X $ such that there is a cover  by open subsets $U\subset X$ with $\rho(U)\subset U$ where $\rho|_U$ reads as 
\begin{equation}\label{standard_Kummer}
(z_1,\dots, z_n)\to (z_1^d,\dots, z_l^d, z_{l+1}  , \dots,  z_n)
\end{equation}
for some choice of local coordinates $(z_1,\dots, z_n)$ with $D$ defined by $z_1\cdots z_l = 0$.
\end{definition}

\begin{notation}
Following \cite{Sabbah}, in the setting of \cref{Kummer_cover}, we will denote the source of $\rho$ by $X_d$ instead of $X$.
\end{notation}

\begin{lemma}[{\cite[Lemma 10.6.10]{Geometric_Stokes}}]\label{goodness_ramified_case_lem}
Let $(X,D)$ be a strict normal crossing pair.
Let $d\geq 1$ be an integer and let $\mathscr{I} \subset \mathscr{IV}_d(X,D)$ be a sheaf.
Then, the following are equivalent:
\begin{enumerate}\itemsep=0.2cm
\item For every $x\in X$, there exist local coordinates $(z_1,\dots, z_n)$ centered at $x$ with $D$ defined by $z_1\dots z_l = 0$ such that for the map $\rho$ given by \eqref{standard_Kummer}, the pullback $\widetilde{\rho}^{\ast}\mathscr{I}$ is a sheaf of unramified irregular values (\cref{goodness_upper_unramified_case}).
\item For every open subset $U\subset X$ and every $d$-Kummer cover $\rho \colon U_d \to U$, the pullback   $\widetilde{\rho}^{\ast}\mathscr{I}$ is a sheaf of unramified irregular values (\cref{goodness_upper_unramified_case}).
\end{enumerate}
\end{lemma}

\begin{definition}[{\cite[Definition 10.6.11]{Geometric_Stokes}}]\label{goodness_ramified_case}
If the equivalent conditions of \cref{goodness_ramified_case_lem} are satisfied, we say that $\mathscr{I} \subset \mathscr{IV}_d(X,D)$ is a sheaf of irregular values.
If furthermore the $\widetilde{\rho}^{\ast}\mathscr{I}$ are good sheaves of unramified irregular values, we say that $\mathscr{I}$ is a good sheaf of  irregular values.
\end{definition}

Analogous statements to \cref{good_sheaf_constructible_finite_stratification} hold in the setting of (ramified) irregular values, as we shall see now.

\begin{proposition}[{\cite[Proposition 10.6.13]{Geometric_Stokes}}]\label{good_sheaf_constructible_finite_stratification_ramified}
Let $(X,D)$ be a strict normal crossing pair.
Suppose that $X$ admits a smooth compactification.
Let $\mathscr{I} \subset \mathscr{IV}_d(X,D)$ be a sheaf of irregular values for some $d\geq 1$.
There exists a finite subanalytic stratification 
$\widetilde{X}\to P$ refining $(\widetilde{X},\widetilde{D})$ such that $\mathscr{I}\in \Cons_P^{hyp}(\widetilde{X},\mathrm{Poset})$.
\end{proposition}

We refer to \cite[Section 10]{Geometric_Stokes} for a more thorough introduction to the classical theory of Stokes functors.

\subsubsection{Compact generators for $\St_{\I, k}$}

\begin{theorem}
Let $(X,D)$ be a strict normal crossing pair.
Suppose that $X$ admits a smooth compactification.
Let $\mathscr{I} \subset \mathscr{IV}_d(X,D)$ be a sheaf of irregular values for some $d\geq 1$.
Then the inclusion $i \colon \St_{\I, k} \hookrightarrow \Repk$ is closed under limits and colimits.
In particular, $\St_{\I, k}$ is presentable stable and the inclusion $i$ admits both a left and a right adjoint.
\end{theorem}

\begin{proof}
This is \cite[Corollary 10.6.14 \& Corollary 7.1.4]{Geometric_Stokes}.
\end{proof}

\begin{recollection}\label{cocartesian_fibration}
	In the setting of \cref{good_sheaf_constructible_finite_stratification_ramified} let $(\widetilde{X},P)$ be a finite subanalytic stratification for which $\pi^{\ast}\mathscr{I} \in \Cons_P^{\hyp}(X; \mathrm{Poset})$.
	Via Lurie's Grothendieck construction (\cite[Theorem 3.2.0.1]{HTT}) and the exodromy equivalence (\cref{exodromy_formal}), the constructible sheaf $\pi^{\ast} \mathscr{I}$ corresponds to a cocartesian fibration in finite posets $\I\to \Pi(\widetilde{X},P)$.
	Since the stratification is finite, $\Pi_\infty(\widetilde{X},P)$ has a finite number of equivalence classes. Since each fiber of the cocartesian fibration is a finite poset, it follows that $\I$ has a finite number of equivalence classes.
\end{recollection}

\begin{definition}
	In the setting of \cref{cocartesian_fibration}, we will refer to $\I \to \Pi(\widetilde{X},P)$ as the cocartesian fibration associated to $\pi^{\ast}\mathscr{I}$.
\end{definition}

\begin{lemma}\label{cpt_gen_Stokes}
	In the setting of \cref{cocartesian_fibration}, let $\left\lbrace c_a \right\rbrace_{a \in I}$ be a finite set of representatives of the equivalence classes of $\I$.
	Let $i^{\mathrm{L}} \colon \Repk \to \St_{\I, k}$ be the left adjoint of the inclusion $i \colon \St_{\I, k} \to \Repk$.
	Then $\bigoplus_{a \in I} i^{\mathrm{L}}(\ev_{c_a, !})$ is a compact generator of $\St_{\I, k}$, where $\ev_{c_a, !} \in \Repk$ is as in \cref{cpt_gen_Rep}.
\end{lemma}

\begin{proof}
	Since the family of functors
\[
\left\lbrace \Hom_{\Repk}(i^{\mathrm{L}}(ev_{c_a,!}), -) \simeq \ev_{c_a} \colon St_{\I, k} \to \Mod_k \right\rbrace_{a \in I}
\]
is jointly conservative, the statement follows.
\end{proof}

\begin{corollary}
	Let $(X,D)$ be a strict normal crossing pair.
	Suppose that $X$ admits a smooth compactification.
	Let $\mathscr{I}$ be a good sheaf of unramified irregular values.
	Let $(\widetilde{X},P)$ be a finite subanalytic stratification for which $\pi^{\ast}\mathscr{I}$ is constructible and $\I \to \Pi_\infty(\widetilde{X},P)$ the associated cocartesian fibration. let $\left\lbrace c_a \right\rbrace_{a \in I}$ be a finite set of representants of the equivalence classes of $\I$.
	For $\Spec(A) \in \dAff_k$, an object $F \in \St_{\I, A} \simeq (\St_{\I, k})_A$ is pseudo-perfect over $A$ if and only if $F(c_a)$ is pseudo-perfect for every $a$.
\end{corollary}

\begin{proof}
	Thanks to \cref{cpt_gen_Stokes}, same as the proof of \cref{pseudo_perf_Rep}.
\end{proof}

\subsection{Good moduli spaces for $\St_{\I, k}$}

	 We show in this paragraph that the moduli of Stokes data admits a good moduli space.

\begin{theorem}[{\cite[Theorem 10.6.15]{Geometric_Stokes}}]\label{Stokes_moduli}
Let $(X,D)$ be a strict normal crossing pair.
Suppose that $X$ admits a smooth compactification and let $\pi\colon \widetilde{X} \to X$ the real blow-up of $X$ along $D$.
Let $\mathscr{I} \subset \mathscr{IV}_d(X,D)$ be a sheaf of irregular values for some $d\geq 1$ and $(X,P)$ be a finite subanalytic stratification for which $\pi^\ast\mathscr{I} \in \Cons_P(\widetilde{X}; \mathrm{Posets})$.
Let $\mathcal{I} \to \Pi_\infty(\widetilde{X},P)$ the cocartesian fibration in finite posets associated to $\pi^\ast\mathscr{I}$.
Then we have an equivalence $\mathbf{St}_{\I, k} \simeq \M_{\St_{\I,k}}$ in $\dSt_k$.
In particular $\mathbf{St}_{\I, k}$ is a locally geometric derived stack, locally of finite presentation over $k$.
\end{theorem}

\begin{notation}
	In the setting of \cref{Stokes_moduli}, we let $\mathbf{St}_{\I, k}^{[0,0]} \subset \mathbf{St}_{\I, k}$ be the open substack corresponding to $\M_{\St_{\I, k}}^{[0,0]} \subset \M_{\St_{\I, k}}$.
	We denote by $\mathbf{St}_{\I, k}^{\heartsuit}$ the truncation of $\mathbf{St}_{\I, k}^{[0,0]}$ .
\end{notation}

\begin{remark}\label{locally_free_valued_Stokes}
	Let $\Spec(A) \in \dAff_k$. In the setting of \cref{Stokes_moduli}, a Stokes functor is $\tau$-flat and pseudo-perfect over $A$ if and only if it takes values in perfect $A$-modules of Tor-amplitude $[0,0]$.
	In particular if $A$ is discrete, a Stokes functor is $\tau$-flat and pseudo-perfect over $A$ if and only if it takes values in locally free $A$-modules of finite rank.
	This follows by combining \cref{t_structure_Stokes} and \cref{cpt_gen_Stokes} with \cref{locally_free_valued}.
\end{remark}

\begin{proposition}\label{openness_flatness_Stokes}
	Let $(X,D)$ be a strict normal crossing pair.
	Suppose that $X$ admits a smooth compactification.
	Let $\mathscr{I} \subset \mathscr{IV}_d(X,D)$ be a sheaf of irregular values for some $d\geq 1$.
	Let $(\widetilde{X},P)$ be a finite subanalytic stratification for which $\pi^{\ast}\mathscr{I} \in \Cons_P(\widetilde{X}; \mathrm{Posets})$.
	Let $\I \to \Pi_\infty(\widetilde{X},P)$ the associated cocartesian fibration.
	The standard $t$-structure $\taust$ on $\St_{\I, k}$ satisfies openness of flatness.
	In particular, $\mathbf{St}_{\I, k}^{\left[0,0\right]}$ is a $1$-Artin derived stack, locally of finite presentation, and $\mathbf{St}_{\I, k}^\heartsuit$ is an algebraic stack locally of finite presentation over $k$.
\end{proposition}

\begin{proof}
Thanks to \cref{Stokes_moduli} and \cref{locally_free_valued_Stokes}, same as the proof of \cref{openness_flatness_Rep}.
\end{proof}

\begin{theorem}\label{Hartogs_Stokes}
In the setting of \cref{Stokes_moduli}, the stack $\mathbf{St}^{\heartsuit}_{\I,k}$ satisfies the Hartogs' principle.
\end{theorem}

\begin{proof}
Let $X = \Spec(A) \in \Aff_k$ be normal integral noetherian and $j\colon U \hookrightarrow X$ the complement of a $0$-dimensional closed subscheme of $X$.
We need to show that every diagram of solid arrows
\[\begin{tikzcd}[sep=small]
	U && {\mathbf{St}^{\heartsuit}_{\I,k}} \\
	\\
	X
	\arrow["F", from=1-1, to=1-3]
	\arrow["j"', hook, from=1-1, to=3-1]
	\arrow[dashed, from=3-1, to=1-3]
\end{tikzcd}\]
By descent, a map $U \to X$ corresponds to a Stokes functor $F\colon \I \to \QCoh(U)^\heartsuit$ that take values in locally free quasi-coherent sheaves of finite rank.
The proof of \cref{Hartogs_Rep} shows that the functor
\[
\mathrm{H}^0(j_\ast ) F \colon \I \to \Mod_A^\heartsuit
\]
\[
c \mapsto \mathrm{H}^0(j_\ast ) (F(c))
\]
is the unique extension of $F$ in $\Rep_{A}(\I)$ that take values in locally free sheaves of finite rank.
We are thus left to check that the above assignment defines a Stokes functor.
Since $j \colon U \to X$ is quasi-compact and quasi-separated, this follows by \cref{push_Stokes_heart}.
\end{proof}

Let $\Spec(\kappa) \to \Spec(k)$ be a closed point with $\kappa$ algebraically closed.
\cref{Hartogs_Stokes} yields the following

\begin{corollary}\label{good_Stokes}
	In the setting of \cref{Stokes_moduli}, the algebraic stack $\mathbf{St}^{\heartsuit}_{\I,k}$ is $\Theta$-reductive and $\mathrm{S}$-complete.
	In particular, any quasi-compact closed substack $\X \subset \mathbf{St}^{\heartsuit}_{\I,k}$ admits a separated good moduli space $X$ whose $\kappa$-points parametrize pseudo-perfect semisimple objects of $\St^{\heartsuit}_{\I,k}$ lying over $\X$.
	Moreover $X$ comes with a natural derived enhancement if $\X$ is the truncation of a derived substack of $\mathbf{St}^{[0,0]}_{\I,k}$.
\end{corollary}

\begin{proof}
By \cref{Hartogs_Stokes} and \cref{Hartogs_implies_existence}, $\mathbf{St}_{\I, k}^{\heartsuit}$ is $\Theta$-reductive and $\mathrm{S}$-complete.
It also has affine diagonal by \cref{affine_diagonal}.
Since these properties are stable under passing to closed substacks, $\X$ satisfies the assumptions of \cref{existence_good_separated}.
The characterization of $\kappa$-points of $X$ follows from \cref{good_moduli_properties}-(4) and \cref{closed_point_is_semisimple}.
The natural derived enhancement are provided by \cref{derived_good}.
\end{proof}

\begin{remark}
The results of \cref{Hartogs_Stokes} and \cref{good_Stokes} holds true (with the same proofs) in the more general setting of \cite[Theorem 8.1.3]{Geometric_Stokes}.
\end{remark}

\subsection{Good moduli for Stokes functors}
	In this subsection we show that good moduli space exists for Stokes data when bounds are imposed on the ranks of the representations.
	We deduce from this that Stokes data admit good moduli spaces.\\
	We will work in the setting of good sheaves of (ramified) irregular values, although all the statements remains true with same proofs in the setting of \cite[Theorem 8.1.3]{Geometric_Stokes}. 

\begin{construction}[{Analogue of \cref{fixed_rank_Rep}}]\label{fixed_rank_Stokes}
	Let $(X,D)$ be a strict normal crossing pair.
	Suppose that $X$ admits a smooth compactification.
	Let $\mathscr{I} \subset \mathscr{IV}_d(X,D)$ be a sheaf of irregular values for some $d\geq 1$.
	Let $(\widetilde{X},P)$ be a finite subanalytic stratification for which $\pi^{\ast}\mathscr{I}$ is constructible and $\I \to \Pi_\infty(\widetilde{X},P)$ the associated cocartesian fibration.
	Let $\left\lbrace c_a \right\rbrace_{a \in I}$ be a finite set of representative of the equivalence classes of $\I$.
	It follows from \cref{locally_free_valued_Stokes} that there is a pullback square
\[\begin{tikzcd}[sep=small]
	{\mathbf{St}_{\mathcal{I}, k}^{[0,0]}} && {\mathbf{St}_{\mathcal{I},k}} \\
	\\
	{\prod_{a \in I}\mathbf{Vect}_k} && {\prod_{a \in I}\mathbf{Perf}_k}
	\arrow[from=1-1, to=1-3]
	\arrow[from=1-1, to=3-1]
	\arrow[from=1-3, to=3-3]
	\arrow[from=3-1, to=3-3]
\end{tikzcd}\]
where the right vertical map is given by the product of evaluations at any $c_a$, $a \in I$, and $\mathbf{Vect}_k \coloneqq \bigsqcup_{n \in \N} \mathrm{B}\mathbf{GL}_n$.
	Choose a set of integers $\underline{r} = \left\lbrace r_a \right\rbrace_{a \in I} \subset \N$ and define the derived stack $\mathbf{St}_{\mathcal{I}, k}^{\underline{r}}$ via the pullback square
\[\begin{tikzcd}[sep=small]
	{\mathbf{St}_{\mathcal{I}, k}^{\underline{r}}} && {\mathbf{St}_{\mathcal{I}, k}^{[0,0]}} \\
	\\
	{\prod_{a \in I} \mathrm{B}\mathbf{GL}_{r_a}} && {\prod_{a \in I}\mathbf{Vect}_k}
	\arrow[from=1-1, to=1-3]
	\arrow[from=1-1, to=3-1]
	\arrow[from=1-3, to=3-3]
	\arrow[from=3-1, to=3-3]
\end{tikzcd}\]
	The stack $\mathbf{St}_{\mathcal{I}, k}^{\underline{r}}$ does not depend on the choice of the $c_a$'s, but only on the choice of the set of integers $\underline{r}$.
\end{construction}

\begin{proposition}\label{quasi_compact_Stokes}
In the setting of Construction \ref{fixed_rank_Stokes}, the derived stack $\mathbf{St}_{\mathcal{I}, k}^{\underline{r}}$ is a quasi-compact open and closed substack of $\mathbf{St}_{\mathcal{I}, k}^{[0,0]}$.
In particular, it is a $1$-Artin stack of finite presentation.
\end{proposition}

\begin{proof}
Same as the proof of \cref{quasi_compact_Rep}.
\end{proof}

\begin{corollary}\label{good_Stokes_rank}
	In the setting of Construction \ref{fixed_rank_Stokes}, the classical truncation $t_0 \mathbf{St}_{\mathcal{I}, k}^{\underline{r}}$ admits a separated good moduli space $t_0 St_{\mathcal{I}, k}^{\underline{r}}$.
	Moreover $t_0 St_{\mathcal{I}, k}^{\underline{r}}$ admits a derived enhancement $St_{\mathcal{I}, k}^{\underline{r}}$ which is a derived good moduli space for $\mathbf{St}_{\mathcal{I}, k}^{\underline{r}}$.
\end{corollary}

\begin{proof}
	It follows directly from \cref{quasi_compact_Stokes} that we can apply \cref{good_Stokes}.
\end{proof}

	Let $\Spec(\kappa) \to \Spec(k)$ be a closed point with $\kappa$ algebraically closed.
	We have the following

\begin{theorem}\label{good_moduli_Stokes_theorem}
	In the setting of \cref{Stokes_moduli}, the algebraic stack $\mathbf{St}^{\heartsuit}_{\I,k}$ admits a separated good moduli space $\underline{\mathrm{St}}^{\heartsuit}_{\I,k}$ whose $\kappa$-points parametrize pseudo-perfect semisimple objects of $\St_{\I, \kappa}^\heartsuit$.
	Moreover $\underline{\mathrm{St}}^{\heartsuit}_{\I,k}$ admits a derived enhancement $\underline{\mathrm{St}}^{[0,0]}_{\I,k}$ which is a derived good moduli space for $\mathbf{St}^{[0,0]}_{\I,k}$.
\end{theorem}

\begin{proof}
	By \cref{good_Stokes_rank}, same as the proof of \cref{good_Rep_theorem}.
\end{proof}

\bibliographystyle{alpha}
\bibliography{good_bib_arxiv}

\end{document}